\documentclass{article}%[12pt]{extarticle}
\usepackage[top=3.5cm,bottom=3.8cm,left=3.3cm,right=3.3cm]{geometry}
\usepackage{amssymb,amsthm,amsmath,amsxtra,dsfont,appendix}
\usepackage[nottoc]{tocbibind}   % add bibliography to table of contents
\usepackage{hyperref} \hypersetup{colorlinks=true, citecolor=blue, linkcolor=black}
\usepackage{mathrsfs}
\usepackage{tcolorbox}
\usepackage{enumitem}
\numberwithin{equation}{section} 

\usepackage[colorinlistoftodos]{todonotes}

\newcommand{\dist}{\mathrm{dist}}

\newcommand{\C}{\mathcal{C}}
\newcommand{\1}{\mathbf{1}}

\newcommand{\E}{\mathbb{E}}
\newcommand{\emp}{\widehat{\mu}}
\newcommand{\En}{\mathscr{E}}
\newcommand{\F}{\mathscr{F}}

\newcommand{\Ha}{\mathscr{H}}
\newcommand{\Hi}{\mathrm{H}}
\renewcommand{\d}{\mathrm{d}}
\newcommand{\g}{\mathrm{g}}

\newcommand{\M}{\mathscr{M}}
\newcommand{\N}{\mathds{N}}
\renewcommand{\O}{\mathcal{O}}
\renewcommand{\P}{\mathbb{P}}
\newcommand{\R}{\mathds{R}}
\renewcommand{\S}{\mathscr{S}}
\newcommand{\Z}{\mathrm{Z}}
\newcommand{\U}{\mathscr{U}}
\newcommand{\W}{\mathscr{W}}
\newcommand{\x}{\mathbf{x}}

\newcommand{\s}{\mathrm{s}}
\renewcommand{\u}{\mathbf{u}}
\newcommand{\T}{\mathds{T}}
\newcommand{\X}{\mathfrak{X}}

\renewcommand\o{
  \mathchoice
    {{\scriptstyle\mathcal{O}}}% \displaystyle
    {{\scriptstyle\mathcal{O}}}% \textstyle
    {{\scriptscriptstyle\mathcal{O}}}% \scriptstyle
    {\scalebox{.7}{$\scriptscriptstyle\mathcal{O}$}}%\scriptscriptstyle
  }

\newtheorem{theorem}{Theorem}[section]
\newtheorem{definition}[theorem]{Definition}
\newtheorem{proposition}[theorem]{Proposition}
\newtheorem{corollary}[theorem]{Corollary}
\newtheorem{lemma}[theorem]{Lemma}
\newtheorem{assumption}[theorem]{Assumptions}

\theoremstyle{definition} \newtheorem{remark}{Remark}[section]

\title{Poisson statistics for Gibbs measures at high temperature}
\date{December 19, 2019}
\author{Gaultier Lambert \footnote{
University of Zurich, Winterthurerstrasse 190, 8057 Z\"urich, Switzerland. 
\newline Email: \href{mailto:gaultier.lambert@math.uzh.ch}{\nolinkurl{gaultier.lambert@math.uzh.ch}}}}

\begin{document}

\maketitle

\begin{abstract}\noindent\normalsize
 We consider a gas of $N$ particles with a general two--body interaction and confined by an external potential in the  mean field or high temperature regime, that is when the inverse temperature $\beta>0$ satisfies $\beta N \to \gamma \ge 0$ as $N\to+\infty$. 
 We show that under general conditions on the interaction and the potential, the local fluctuations are described by a Poisson point process in the large $N$ limit. 
 We present applications to Coulomb and Riesz gases on $\R^n$ for any $n\ge 1$, as well as to the edge behavior of $\beta$--ensembles on $\R$. 
\end{abstract}

%	60B20, 60G70, 60G55, 60F10 
%  	Random matrices, Extreme value theory, point processes, large deviations

%\tableofcontents
%\clearpage

 \section{Introduction}

The present article is inspired by the works of \cite{BP15} and \cite{NT} which establish that the local statistics of a $\beta$--ensembles on $\R$ are Poisson in the regime as $\beta N \to \gamma$ for a fixed $\gamma>0$. 
Recall that a $\beta$--ensembles is a collection of particles with random positions in $\R^N$ having a joint distribution
\begin{equation} \label{GbE}
\P_N  = \frac{ e^{- \Ha_N(\x)}}{\Z_N^{(V)}} \d \x ,  \qquad  \Ha_N(\x)=  \beta \sum_{1\le i <j \le N} \log|x_i-x_j|^{-1} +  \sum_{1\le j \le N} V(x_j)  , 
\end{equation}
where  $\d \x = \d x_1 \cdots \d x_N$ and $\Z_N^{(V)}>0$ is a normalization constant. 
$ \Ha_N(\x)$ corresponds to the energy of a configuration $\x\in \R^N$ and it is made of a two--body interaction with kernel $\g(x,z) = \log|x-z|^{-1}$ and a one--body potential~$V(x)$. 
The specificity of this interaction lies in the singularity of its kernel on the diagonal which forces the particles to repel each other.
The parameter $\beta>0$ can be thought of both as a \emph{coupling constant} which represents the strength of the interaction and as an \emph{inverse temperature} if one views $\P_N$ has a Gibbs measure (as $\beta$ increases, the measure $\P_N$ concentrates on the low energy configurations). 
In the Gaussian case, $V(x) = \alpha x^2$, $\P_N$  describes the joint law of the eigenvalues of a symmetric tri--diagonal random matrix known as the Dumitriu--Edelman ensemble \cite{DE02}. 
In \cite{BP15},  Benaych-Georges and P{\'e}ch{\'e} used the Dumitriu--Edelman representation to show that the local fluctuations inside of the bulk of the eigenvalues' process are described by a homogeneous Poisson point process on $\R$ as $N\to+\infty$ with  $\beta N \to \gamma \ge 0$. 
This result has been recently generalized by  Nakano and Trinh \cite{NT}  to any potential $V\in C(\R)$ with sufficient growth, hence showing universality. 
In this article, we tackle an analogous problem for a general (singular) interaction $\g$ on a $n$--dimensional manifold $\X$. 
In particular, our results apply to any Riesz gas on $\R^n$ for any $n\ge1$ with a general potential (Corollary~\ref{cor:RBF}) and to Coulomb gases on compact manifolds of dimension $n\ge 2$. We also discuss the emergence of Poisson statistics at the \emph{boundary} of the gas. 
This allows us to show that the largest eigenvalue of the Dumitriu--Edelman ensemble properly rescaled converges to a Gumbel random variable in the high temperature regime (see Corollary~\ref{thm:Gumbel}).

%This constant $C$ is allowed to depend on $\gamma \ge 0$, the potential~$V$ and the dimension~$d$. 

\subsection{Model} \label{sect:model}

Let $\X$ be a connected differentiable manifold of (real) dimension $n$ equipped with a Riemannian metric~$g$. Then
$\X$ is a metric measured space equipped with its geodesic distance $\dist(\cdot)$ and with its Borel $\sigma$--algebra and volume density $\omega$.
In local coordinates, this density can be expressed as $\d\omega = \sqrt{\det g}\, \d x_1 \cdots \d x_n$. 
 We denote by $\M(\X)$ the set of probability measures on $\X$ equipped with the weak topology. Then, $\M(\X)$ is a Polish space.  If $\X$ is compact, by scaling, we also assume that  $\omega(\X) =1$.

\medskip

We consider a gas of $N$ particles interacting via a two--body kernel $\g : \X \times \X \to (-\infty, +\infty]$ and confined by a potential $V : \X \to (-\infty, +\infty]$ at inverse temperature $\beta > 0$. That is for a given $N\in\N$, if $\x=(x_1,\dots, x_N) \in \X^N$ is a configuration, we define its energy:
\begin{equation} \label{def:H}
\Ha_N(\x)  := \beta \sum_{1\le i <j \le N} \g(x_i ,x_j) +  \sum_{1\le j \le N} V(x_j)  .
\end{equation}
Then, we consider the \emph{Gibbs measure} on $\X^N$:
\begin{equation} \label{def:Gibbs}
\P_N[\d \x ]  : =  \frac{ e^{- \Ha_N(\x)}}{\Z_N^{(V)}}  \prod_{j=1}^N \omega(\d x_j), 
\end{equation}
where the partition function
\begin{equation} \label{def:Z}
\Z_N^{(V)} := \int e^{- \Ha_N(\x) } {\textstyle \prod_{j=1}^N} \omega(\d x_j)   < +\infty . 
\end{equation}
In this article, we are interested in describing the local fluctuations of a random configuration sampled from $\P_N$  in the so--called \emph{mean field regime} where the coupling constant $\beta=\O(N^{-1})$ in the large $N$ limit. 
This is also known as the \emph{high temperature regime}, in contrast with the case where $\beta>0$ is fixed.
Our main results show that for a large class of interactions and confining potentials, as $\beta N \to \gamma$ with $\gamma>0$, the local statistics are Poisson in the bulk as well as near the boundary of the configuration of particles. 
The thermodynamical limit in the regime where  $\beta$ is fixed and the potential  $V = N \mathrm{V}$ is harder to analyze because of the strong correlations between the particles. 
Nevertheless, it has been established that the local fluctuations of the Gaussian $\beta$--ensembles on $\R$ are described by the Sine$_\beta$ point process in the bulk \cite{KS09, VV09} and by the  Airy$_\beta$ point process at the edges \cite{RRV11}.  The question of universality  of these limits for different potentials has also been tackled in \cite{BEY13, BEY14}. 
Recently, there have also been several considerable advances to describe the  thermodynamical limit for Coulomb and Riesz gases on $\R^d$ with $d\ge 2$.
In particular, Lebl\'e and Serfaty  showed that the  local statistics are described by point processes which minimize of a certain  \emph{free energy functional} introduced in \cite{LS17}. 
A local law, as well as precise estimates for the fluctuations have also been obtained by Amstrong and Serfaty in \cite{AS}. We refer to the review article \cite{Serfaty1} for a comprehensive overview of these developments and the appropriate references. 
On the other hand, if $\beta =0$,  the particles are independent and all the large $N$ properties of the gas can be inferred from the classical theory of independent random variables. 
The high temperature regime is of particular interest since it  interpolates between the two aforementioned cases and because of the competition between the energy and the entropy of the gas, see \cite{AB19, HL}. 

\medskip 

In order guarantee that the condition \eqref{def:Z} holds and that the Gibbs measure satisfies a large deviation principle as $N\to+\infty$, one needs to impose certain regularity and growth conditions on the  two--body kernel $\g$ and the potential $V$ of the model. 
We will work under the following general assumptions which  allow for singular interactions in order to model the repulsion between the particles. 

\begin{assumption} \label{ass:1}
 We suppose that the kernel $\g : \X \times \X \to [0, +\infty]$ and  the potential $V : \X \to[0, +\infty]$ satisfy the following properties:
\begin{enumerate}[leftmargin=.5cm]
\item[\rm 1)] $V$  is continuous\footnote{The set $\S_V$ is open and we assume it is not empty. Moreover, we set $e^{-V(x)}\equiv 0$ for all $x\in \{V=+\infty\}$.}  on $\S_V =\{V<+\infty\}$ and  $ \displaystyle \int e^{- V(x)} \omega(\d x) <+\infty$.
\item[\rm 2)]$\g$ is symmetric and lower semicontinuous. For any $\epsilon>0$, there exists $\delta>0$ such that for all $u,x\in\X$,  $ |\g(u,x)| \le \epsilon\, \dist(x,u)^{-n}$ if $\dist(u,x) \le \delta$. 
\item[\rm 3)]
For any $k\in\N$, we can decompose $\g = \g_k + \g^k$ where $\g_k \le k$ is continuous, $\g^k \ge 0$ and there exists   $p>1$ and a sequence $\mathrm{c}_k \to 0$ as $n\to+\infty$ such that   $ \displaystyle \int \g^k(u,x)^p\, \omega(\d x) \le \mathrm{c}_k$ for all $u \in \X$. 
\end{enumerate}
\end{assumption}

The main example of interaction which satisfies the Assumptions~\ref{ass:1}.2)3) are the Riesz kernels: for any $n\ge1$ and $\s \in (0, n)$, 
\begin{equation} \label{Riesz}
\g_\s(u,x) =  |x-u|^{-\s} , \qquad x , u\in\R^n .
\end{equation}
If $n\ge 3$ and $\s = n-2$, up to a constant, this corresponds to the Coulomb kernel. Then, the energy \eqref{def:H} is that of a gaz of $N$ electric charges at inverse temperature $\beta>0$ confined by the background potential $V$. 

\medskip

Observe that since we assume that $\g\ge 0$, the condition 1) implies that the partition function \eqref{def:Z}  is finite for all $N\in\N$, 
\begin{equation} \label{ZUB}
\Z_N^{(V)} \le \int e^{-\sum_{j=1}^NV(x_j) } {\textstyle \prod_{j=1}^N} \omega(\d x_j)  \le  \mathrm{C}_0^{N} 
\qquad\text{where}\quad
 \mathrm{C}_0 :=  \int e^{- V(x)} \omega(\d x) .
\end{equation}
Let us also record that the  condition 3) implies that the kernel $\g$ is locally integrable in the sense that for any compact set $\mathcal{K} \subset \X$ and for all $u \in\X$,  
$\displaystyle \int_{\mathcal{K}} \g(u,x) \omega(\d u) <+\infty$.

\medskip

If $\g$ is not bounded from below,  as it the case for the two--dimensional Coulomb kernel, 
\begin{equation} \label{def:log}
\g(u,x) = \log |x-u|^{-1} ,
\end{equation}
 then we need to modify slightly our assumptions. Let us observe that for any continuous function $\vartheta \ge 0$, we can always rewrite  the energy \eqref{def:H} as 
\begin{equation} \label{H2}
\Ha_N(\x)  = \beta \sum_{1\le i <j \le N} \widetilde\g(x_i ,x_j) +  \sum_{1\le j \le N} \widetilde{V}(x_j) 
\end{equation}
where  $\widetilde{V}(x) = V(x)- \beta(N-1)  \vartheta(x) $ and $\widetilde\g(u ,x) =  \g(u,x) + \vartheta(u) + \vartheta(x)$. The idea is to choose the function $\vartheta$ in such a way that for all $u,x\in\X$, 
\begin{equation} \label{def:vartheta}
-\g(u,x) \le  \vartheta(u) + \vartheta(x) ,
\end{equation}
Moreover, if $V$ grows sufficiently quickly, then we can also always assume that $\widetilde{V} \ge 0$ for all $N\in\N$ (since $\beta N =\O(1)$ and after possibly adding a constant to $V$ which does not change the Gibbs measure \eqref{def:Gibbs}).
In this setting, let us state our assumptions. 

\begin{assumption} \label{ass:2}
 We suppose that the kernel $\g : \X \times \X \to(-\infty, +\infty]$ satisfies the properties $2)$ and $3)$ from Assumptions~\ref{ass:1} and that there exists a continuous function $\vartheta \ge 0$ so that \eqref{def:vartheta} holds. Then, we suppose that the potential $\widetilde{V}_\kappa = V - \kappa \vartheta$ is bounded from below and satisfies the properties $1)$ for all $\kappa \ge 0$. 
\end{assumption}

\paragraph{Notation.} In the following, we use the notation $A \ll B$ if there exists a constant $C>0$ independent of $N$ such that $|A| \le C B$. 
We also use the notation $A \asymp B$ when  $A \ll B$ and $B\ll A$. 
For any $n\in\N$, we denote by $|\cdot|$ the Euclidean norm on $\R^n$. 
For any probability measure $\mu$ on $\X$, if it exists, we denote by $\mu(x)$ its density function.

\subsection{Law of large numbers} \label{sect:lln}

Let us review some of the basic properties of the particle system defined in the previous section. 
A convenient way to encode a configuration $\x \in\X^N$ of particle is through its empirical measure:
\[
\emp_N^{(\x)}:= N^{-1} {\textstyle \sum_{j=1}^N}  \boldsymbol\delta_{x_j} . 
\]
We usually ignore the superscript $(\x)$ and view $\emp_N$ as a random measure under $\P_N$. 
Let us now recall how to  describe the equilibrium properties of the gas. We  define the \emph{energy functional}:
\begin{equation} \label{def:E}
\En(\mu) := \iint \g(x,z) \mu(\d z)\mu(\d x)  , \qquad \mu\in \M(\X).  
\end{equation}
The interpretation is that $\mu$ represents a cloud of particles and $\En(\mu)$ is the self--energy of this cloud.  
 Moreover, the  potential generated by $\mu\in \M(\R^d)$ will be denoted by 
\[
\U^\mu(x)  : =  \int \g(x,z) \mu(\d z)  , \qquad x\in\X.
\]  

If the kernel  $\g\ge0$, then $\En(\mu)$ and $\U^\mu(x)$ are well--defined for all  $\mu\in \M(\R^d)$. On the other--hand, when $\g$ is not bounded from below, we can consider instead the  \emph{weighted energy}:
 \begin{equation} \label{def:Ew}
\widetilde\En(\mu) := \iint \widetilde\g(x,z) \mu(\d z)\mu(\d x)  , \qquad \mu\in \M(\X), 
\end{equation}
where $\widetilde\g(u ,x) =  \g(u,x) + \vartheta(u) + \vartheta(x)$ as in \eqref{def:vartheta}. 

\medskip

Recall that for any $\nu \in \M(\X)$,  the relative entropy (or Kullback--Leibler divergence) with respect to $\nu$ is defined by
\begin{equation}   \label{def:ent}
 \Hi(\mu|\nu)  =  \begin{cases} \displaystyle \int \log(\rho)\, \d\mu  &\text{if } \d\mu = \rho\, \d\nu \\
 +\infty   &\text{else} 
 \end{cases} , \qquad \mu\in\M(\X) . 
\end{equation}
The function $\mu \mapsto \Hi(\mu|\nu)$  can be though of as a \emph{distance} from $\mu$ to $\nu$ as can be seen from Lemma~\ref{lem:entropy}.
The reference measure that we consider are
\begin{equation}  \label{def:mu0}
\nu_\gamma(\d x) = \mathrm{C}_\gamma^{-1} e^{-\widetilde{V}_\gamma(x)} \omega(\d x) 
\qquad\text{where}\quad
\widetilde{V}_\gamma = V - \gamma \vartheta \quad\text{for}\quad \gamma \ge 0. 
\end{equation}
Under the Assumptions~\ref{ass:2}, we can choose $\mathrm{C}_\gamma>0$ (increasing) such that $\nu_\gamma$ is a probability measure. The relevant functional to describe the equilibrium configuration of our particle system as $\beta N\to\gamma$ and $N\to+\infty$ is the \emph{free energy}: 
\begin{equation} \label{def:F}
 \F_\gamma(\mu) = \frac\gamma2 \widetilde\En(\mu) + \Hi(\mu|\nu_\gamma)  , \qquad  \mu\in  \M(\X) \text{ and }\gamma \ge 0.
 \end{equation}
This might seem like a slightly unusual way to define the free energy, but observe that it has the following properties: $ \F_\gamma \ge 0$ and it is lower--semicontinuous, so that it attains its minimum and all minimizers are absolutely continuous with respect to $\nu_\gamma$ (and a fortiori $\nu_0$). 
Moreover, it follows from formula \eqref{F2} below  that there exists a constant $c_\gamma>0$ such that for any $\mu\in\M(\R^n)$ with a density $\mu \ll \nu_0$, 
\begin{equation} \label{F3}
 \F_\gamma(\mu) = \frac\gamma2  \En(\mu) + \Hi(\mu|\nu_0) + c_\gamma <+\infty . 
\end{equation}

\medskip

The following Law of large numbers for the empirical measure can be extracted from the literature (it basically  follows from the  large deviation principle in~\cite{GZ18}). 

   \begin{proposition} \label{prop:eq}
   Let us suppose that the  Assumptions~\ref{ass:2} hold and that the free energy $\eqref{def:F}$ has a {\bf unique} minimizer denoted by $\mu_\gamma \in \M(\X)$. 
Then under $\P_N $, the empirical measure $\widehat{\mu}_N$ converges in probability to $\mu_\gamma$ as $\beta N\to \gamma$ and $N\to+\infty$.  
Moreover, the {\bf equilibrium  measure}  $\mu_\gamma$ satisfies
\begin{equation}  \label{eq:mu}
\mu_\gamma(\d x) = \mathrm{L}_\gamma^{-1} e^{- \gamma \U^{\mu_\gamma}(x) - V(x)} \omega(\d x)  , \qquad  \mathrm{L}_\gamma  >0 . 
\end{equation}
\end{proposition}

For completeness, we review the important steps of the proof of Proposition~\ref{prop:eq} in Section~\ref{sect:eq} of the Appendix.
In particular, we carefully derive the self--consistent equation \eqref{eq:mu} which characterizes the minimizer(s) of  $\F_\gamma$ (see Proposition~\ref{prop:regularity}).
It turns out that under the assumptions of Proposition~\ref{prop:eq}, the equilibrium potential $\U^{\mu_\gamma}$ is continuous on $\X$ so that the equation \eqref{eq:mu} is satisfied by the equilibrium density for all $x\in\X$.  
 Proposition~\ref{prop:eq} implies that for any function $f\in \C(\X)$ which is uniformly bounded and for any $\epsilon>0$,
\begin{equation} \label{LLN}
\lim_{N\to+\infty} \P_N \bigg[ \bigg|  \int f \d\emp_N - \int f \d\mu_\gamma \bigg| \ge \epsilon \bigg] =0 . 
\end{equation}

\medskip

If the limiting temperature $\gamma=0$, the equilibrium density is $\mu_0(x)=  \mathrm{L}_0^{-1} e^{- V(x)}$ and it is the (unique) minimizer of the relative entropy $\Hi(\cdot|\mu_0) $, so our notation are consistent. 
Note that if the kernel $\g \ge 0$,  then $\vartheta=0$ and $\nu_\gamma = \mu_0$. 
Then, we inferred from the equation \eqref{eq:mu} that   $0\le \U^{\mu_\gamma}  \ll 1 $, so that the equilibrium density satisfies $\mu_\gamma(x) \asymp e^{-V(x)}$ for all $x\in\X$ and for any $\gamma\ge 0$. 
In fact, under the Assumptions~\ref{ass:2}, we easily obtain he estimate for all $x\in \X$, 
\begin{equation}  \label{muk:tail}
e^{-V(x)} \ll  \mu_\gamma(x) \ll  e^{-\widetilde{V}_\gamma(x)} ,
\end{equation}
see \eqref{Uest} below. 
Finally, the issue about the uniqueness of the minimizer of the free energy in the case of the Riesz \eqref{Riesz} and log \eqref{def:log} gases is addressed at the end of Section~\ref{sect:eq}. 

\subsection{Local fluctuations}

In this section, we present our main result concerning the local fluctuations of the system of particles defined in Section~\ref{sect:model}.
Let us fix a point $E \in \S_V$ and let $U$ be a normal neighborhood of $E$ in  $\S_V$  and $\varphi : U\to\R^n$ denotes a normal coordinates chart. Then, we define the \emph{local point process} around $E$ as 
 \begin{equation} \label{def:Xi} 
 \Xi_N : = \sum_{x_j \in U}  \boldsymbol\delta_{N^{1/n} \varphi(x_j)} . 
 \end{equation}
Our main result states that $\Xi_N$ converges in distribution to a  (homogeneous\footnote{
The reference measure on $\R^n$ is the Lebesgue measure and we refer to Definition~\ref{def:Poisson} in the appendix for the definition of a Poisson point process. In fact, one could use any $\C^1$ chart to define the local point process~\eqref{def:Xi}, then the limiting Poisson process would also be homogeneous but its intensity need not be given by the equilibrium density $\mu_\gamma(E)$ at $E\in \S_V$.}) Poisson point process on $\R^n$ as the number of particles $N\to+\infty$ and  $N\beta \to \gamma$. 

 \begin{theorem} \label{thm:LF}
Fix $\gamma\ge 0$ and suppose that the  Assumptions~\ref{ass:2} hold and that  the free energy \eqref{def:F} has a {\bf unique} minimizer $\mu_\gamma \in \M(\X)$. Then, under $\P_N$, $\Xi_N$ converges in distribution as $N\to+\infty$ and  $N\beta \to \gamma$  to a homogeneous Poisson process on $\R^n$ with intensity $\mu_\gamma(E) >0$. 
 \end{theorem}

 Theorem~\ref{thm:LF} confirms that as expected, Poisson statistics describe  the local fluctuations of a Coulomb gas in the mean field regime. However,  by analogy with the 1-dimensional case studied in~\cite{HL},  we still expect some correlations between the particles in the sense that the global fluctuations of the equilibrium remain non--trivial\footnote{We mean that these global fluctuations are not described by a \emph{white noise} as it is the case when $\beta=0$.} as $N\beta \to \gamma >0$. 
Moreover, in view of the result from \cite{AD14b}, one expects that Poisson statistics describes the microscopic behavior of  $\beta$--ensembles regardless of the confining potential as long as the inverse temperature $\beta\to0$ as $N\to+\infty$.

\medskip

Observe that it follows from the estimate \eqref{muk:tail} that $\mu_\gamma(E) >0$ if and only if   $E \in \S_V$, so that the limiting process is non trivial. 
The proof of Theorem~\ref{thm:LF} is given in Section~\ref{sect:LF} and it relies only on the Assumptions~\ref{ass:2} for our particle system and the Law of large numbers from Proposition~\ref{prop:eq}. 
For completeness, we review briefly the concept of convergence for point processes in Section~\ref{sect:pp}. 

\paragraph{Coulomb gas in the Euclidean case.} 
Our main applications of  Theorem~\ref{thm:LF} is to deduce universality of local fluctuations for Coulomb and Riesz gases on $\R^n$ for any $n\ge 1$. 
Let us recall that the Riesz kernels on $\R^n$ are given by  $\g_\s(u,x) =  |x-u|^{-\s}$ for an exponent  $\s \in (0, n)$. Then, we consider the Gibbs measure \eqref{def:H}--\eqref{def:Gibbs} on $(\R^n)^N$ with interaction kernel $\g_\s$ and a general potential $V :\R^n \to[0, +\infty]$ such that $e^{- V} \in L^1(\d x)$ where $\d x$ denotes the Lebesgue measure. 
In this case, we obtain the following result which describes the local fluctuation of the system of particles in the bulk as $N\beta \to \gamma$. 

\begin{corollary}[Riesz gases -- bulk fluctuations] \label{cor:RBF}
If $V$  is continuous on $\S_V =\{V<+\infty\}$, then for any fixed $E\in \S_V$, the  point process $\Xi_N =  \sum_{j=1}^N  \boldsymbol\delta_{N^{1/n} (x_j-E)}$ converges in distribution as $N\to+\infty$ and  $N\beta \to \gamma$ with $\gamma \ge 0$  to a  homogeneous Poisson point process on $\R^n$.
\end{corollary}
 
 The proof of Corollary \ref{cor:RBF} consists in verifying that the free energy \eqref{def:F} has a unique minimizer and that the Riesz kernels $\g_\s$ satisfy the Assumptions~\ref{ass:1}.2)3) if we set
$\g_{\s,k} = \g_s \wedge k $  for $k\in\N$. 
Since $\g_\s \ge 0$ is a positive--definite translation--invariant kernel, the free energy $\F_\gamma  = \frac\gamma2  \En + \Hi(\cdot|\mu_0)$ is strictly convex, so that it has a unique minimizer -- see Lemma~ \ref{lem:Rminimizer} in the Appendix for a precise claim. 
Moreover, the Assumptions~\ref{ass:1}.2) clearly holds and the function $\g_{\s,k}$ are continuous and it holds for all $k\in \N$, $p < n/\s$  and $u\in\R^n$, 
 \begin{equation} \label{intcondition}
  \int \g_\s^k(u,x)^p \d x  = \int_{|x| \le k^{-1/\s}} |x|^{-p\s} \d x =  \Omega_n k^{-(n/\s-p)} ,
 \end{equation}
 where $\Omega_n$ denotes the volume of the unit ball in $\R^n$. 
 
 \medskip
 
 We also obtain the analogous result for log gases in dimensions $n=1,2$. Let us denote  for $x,u\in \R^n$, 
 \begin{equation} \label{logkernel}
 \g(x,u) = \log|x-u|^{-1} 
 \qquad\text{and}\qquad
 \vartheta(x) = \log(1+|x|) , 
 \end{equation}
so that the condition \eqref{def:vartheta} holds.
Like \eqref{intcondition}, we easily verify that the log kernel $\g$ satisfies the Assumptions~\ref{ass:1}.2)3)  
if we set $\g_k = \g \wedge k $  for $k\in\N$. 
Moreover, by Lemma~ \ref{lem:Lminimizer} in the Appendix, in dimensions 1 and 2, the free energy \eqref{def:F} has a unique minimizer $\mu_\gamma \in \M(\X)$ for any $\gamma \ge 0$. 
Thus,  we obtain the following result directly from Theorem~\ref{thm:LF}. 
 
 \begin{corollary}[Log gases -- bulk fluctuations] \label{cor:LBF}
Suppose that $V$ is continuous on the set $\S_V =\{V<+\infty\}$ and that for all $\kappa \ge 0$, 
\begin{equation} \label{c3}
\inf_{x\in \S_V}\big\{ V(x) -  \kappa \vartheta(x) \big\} >-\infty
\qquad\text{and}\qquad
\int |x|^{\kappa} e^{- V(x)}\d x  <+\infty  . 
\end{equation}
Then, for any $E\in \S_V$, the  point process $\Xi_N =  \sum_{j=1}^N  \boldsymbol\delta_{N^{1/n} (x_j-E)}$ converges weakly  as $N\to+\infty$ and  $N\beta \to \gamma$ with $\gamma \ge 0$  to a  homogeneous Poisson point process on $\R^n$.
\end{corollary}
  
  \paragraph{Coulomb gas on compact Riemannian manifolds.}
  This particle system has been introduced in \cite{GZ18} and further studied in \cite{GZ19}.
  It consists of a Gibbs measure \eqref{def:H}--\eqref{def:Gibbs} where the interaction kernel $\g$ is the Green function of the Laplace--Beltrami operator on a compact manifold $\X$ of dimension $n\ge 3$. This means that in distributional sense, $\Delta \g(\cdot,x) = -\boldsymbol\delta_x +1$ for every $x\in\X$. 
  The Green function $\g$ is symmetric, lower semicontinuous and it satisfies $\int \g(x,u) \omega(\d u) =0$ for all $u\in\X$. Important other properties (which relies on the compactness of $\X$ -- see \cite[Chapter 4]{Aubin98}) include that $\g \ge -c_{\X}$ for a constant $c_{\X}$ and that
  \begin{equation} \label{est:g}
  \g(x,u) \ll \dist(x,u)^{2-n} . 
  \end{equation}
 Moreover, it holds for any function $f\in L^2(\X)$,
  \[
0 \le \En(f)= \iint \g(x,z)  f(x) f(z) \omega(\d z)\omega(\d x)   \ll \|f\|_{L^2(\X)}^2 , 
  \]
 This follows from the fact that the kernel $\g$ defines a positive compact operator on the Hilbert space $L^2(\X)$. In fact $ \En(f)=0$ if and only if $f$ is constant. 
 By \eqref{convexity}, this implies that the energy functional $\En$ is strictly convex on  $\M(\X)$ and that its unique minimizer is the volume density~$\omega$. 
Then, by Lemma~\ref{lem:entropy}, the free energy $\F_\gamma(\mu) = \frac\gamma2  \En(\mu) + \Hi(\mu|\mu_0)$ is also strictly convex, so that it has a unique minimizer $\mu_\gamma \in \M(\X)$.
 Since $\g$ is bounded from below and $\X$ is compact, we can consider the case where $V=0$
  in which case the equilibrium measure $\mu_\gamma = \omega$ for all $\gamma\ge 0$. 
Using the estimate \eqref{est:g}, it is immediate to verify that $\g$ satisfies the Assumptions~\ref{ass:1}.2)3). Hence, by Theorem~\ref{thm:LF}, for any $E\in\X$, the point process \eqref{def:Xi} converges as $N\to+\infty$ and  $N\beta \to \gamma$ to  homogeneous Poisson process on $\R^n$ with intensity 1. 
  
\subsection{Boundary fluctuations}
 
 In this section, we focus on the Euclidean case $\X = \R^n$ for $n\ge1$ and for simplicity we consider the potential   $V(x) = |x|^\alpha$ for a $\alpha>0$. 
 In this case, since the equilibrium measure $\mu_\gamma$ is radial and decays at $\infty$, the particles fill in a ball with a large radius and we can study the local fluctuations in the vicinity of the boundary of this ball for large $N$. Then, we typically expect to observe an inhomogeneous Poisson point process. 
 By adapting the proof of  Theorem~\ref{thm:LF}, we obtain  the following result for Riesz gases. 
 
 \begin{theorem}[Riesz gases -- boundary fluctuations] \label{thm:REF}
 Suppose that  $V(x) = |x|^\alpha$ for $x\in\R^n$ and  $\alpha>0$. For any $ \upsilon \in \mathbb{S}^{n-1}$, let
 $\psi \in SO(n)$ such that $\psi(e_1)=v$ and set 
 $\varphi_N(x) =  \eta_N\big(\upsilon + \alpha^{-1}\eta_N^{-\alpha}\psi(x) \big) $
 where
 \begin{equation} \label{def:eta1}
 \eta_N = (\log N)^{1/\alpha} \Big( 1 - \frac{n(\alpha-1)}{\alpha^2}  \frac{\log\log N}{\log N} - \frac{\log(\alpha^n \mathrm{L}_\gamma )}{\alpha \log N} \Big) 
 \end{equation}
  and $ \mathrm{L}_\gamma$ is as in \eqref{eq:mu}. 
 Then, the point process $\Xi_N = \sum_{j=1}^N  \boldsymbol\delta_{ \varphi_N^{-1}(x_j)}$  obtained by zooming around the point~$\eta_N\upsilon$  converges in distribution as $N\beta \to \gamma$ and $N\to+\infty$ to a Poisson point process with intensity $\theta(x) = e^{-e_1 \cdot x}$ on $\R^n$. 
 \end{theorem}
 
 Theorem~\ref{thm:REF} follows from a result, Theorem~\ref{thm:limitcase}, given in Section~\ref{sect:edge} which can be applied to any non--negative interaction kernel which decays at $\infty$ and to a general potential~$V$. 
 However, since the assumptions of Theorem~\ref{thm:limitcase} are technical and depend strongly on the growth of $V$ at $\infty$, we have decided to focus on a concrete example in this section. 
 It is should be noted that the difference with the case of independent particles $(\gamma=0)$ come only from the last term in the definition of the \emph{radius} \eqref{def:eta1} where the constant $\mathrm{L}_\gamma$ from the self--consistent equation for the equilibrium density appear instead of $\mathrm{L}_0$. 
 This is a consequence of the fact that the Riesz kernel \eqref{Riesz} decays to zero at $\infty$. 
For the log kernel \eqref{def:log} in dimension $n=1,2$, we obtain an analogous result but the \emph{radius} $\eta_N$ of the ball where the particles are \emph{confined} for large $N$ depends on the temperature $\beta$.

  \begin{theorem}[Log gases -- boundary fluctuations] \label{thm:LEF}
 Suppose that  $V(x) = |x|^\alpha$ for $x\in\R^n$ and $\alpha>0$. For any $ \upsilon \in \mathbb{S}^{n-1}$, let
 $\psi \in SO(n)$ such that $\psi(e_1)=v$ and set 
 $\varphi_N(x) =  \eta_N\big(\upsilon + \alpha^{-1}\eta_N^{-\alpha}\psi(x) \big) $
 where
 \begin{equation} \label{def:eta2}
 \eta_N = (\log N)^{1/\alpha} \Big( 1 + \frac{ \beta N  - n(\alpha-1)}{\alpha^2}  \frac{\log\log N}{\log N} - \frac{\log(\alpha^n \mathrm{L}_\gamma )}{\alpha \log N} \Big)
 \end{equation}
 and $ \mathrm{L}_\gamma$ is as in \eqref{eq:mu}. 
 Then, the point process $\Xi_N = \sum_{j=1}^N  \boldsymbol\delta_{ \varphi_N^{-1}(x_j)}$  obtained by zooming around the point~$\eta_N\upsilon$  converges in distribution as $N\beta \to \gamma$ and $N\to+\infty$  to a Poisson point process with intensity $\theta(x) = e^{-e_1 \cdot x}$ on $\R^n$. 
 \end{theorem}
 
 The proof of Theorem~\ref{thm:LEF} is also given in Section~\ref{sect:edge} (see Theorem~\ref{thm:logedge}).
Let us observe that since 
 $\varphi_N(x) =  \eta_N\upsilon + \alpha^{-1}  (\log N)^{1/\alpha-1}\psi(x)\big(1+\O(\frac{\log\log N}{\log N}) \big)$, it is not possible to replace $\beta N$ by $\gamma$ on the RHS of \eqref{def:eta2} without information on the rate of convergence as $N\beta \to \gamma$.

 \medskip

Theorem~\ref{thm:LEF} with $n=1$ and $\alpha=2$ describes the fluctuations for Gaussian $\beta$--ensemble  \eqref{GbE} near the edges in the regime  $\beta N \to \gamma$ with $\gamma \ge 0$ as $N\to+\infty$. 
In particular observe that for any $t\in\R$,
\[ \begin{aligned}
\P_N\big[ \Xi_N(t,+\infty) =0 \big]  
&= \P_N\big[\text{no particle }> \eta_N + \tfrac{t}{2\eta_N} \big]  \\
&= \P_N\big[ \max_{j=1,\dots, N} x_j \le  \eta_N + \tfrac{t}{2\eta_N} \big] . 
\end{aligned}\]
Since the random variable $\Xi_N(t,+\infty)$ converges to a Poisson random variable with mean $\lambda(t) = \int_t^{+\infty} e^{-x} dx$, this implies that for any $t\in\R$, 
\[
\lim_{N\to+\infty} \P_N\big[ \max_{j=1,\dots, N} x_j \le  \eta_N + \tfrac{t}{2\eta_N} \big] 
= \exp\big(-\lambda(t)\big) = \exp(-e^{-t}) . 
\]
This shows that in the high temperature regime, the largest eigenvalue of the  Gaussian $\beta$--ensemble suitably normalized converges to a Gumbel random variable.

 \begin{corollary}[Edge fluctuations for Gaussian $\beta$--ensemble] \label{thm:Gumbel}
 Consider the ensemble \eqref{GbE} with $V(x) =x^2$ in the regime where $\beta N \to \gamma$ with $\gamma \ge 0$ as $N\to+\infty$ and let 
\[ 
\xi_N := 2 \Big( \sqrt{\log N}  \max_{j=1,\dots, N} x_j -  \log N \Big) - \tfrac{ \beta N -1}{2} \log\log N + \log(2\mathrm{L}_\gamma) 
\] 
where $\mathrm{L}_\gamma>0$ is as in formula \eqref{eq:mu}. 
Then the random variable $\xi_N$ converges in distribution  to a standard Gumbel random variable $($with distribution function $\exp(-e^{-t})$ on $\R)$.
 \end{corollary}
 
 The question of the fluctuations at the edge of the Gaussian $\beta$--ensemble in the high temperature regime has been raised in \cite{AD14a}.
 Then, it was proved in \cite[Theorem 1.1]{Pakzad18} that if $\beta N \ll \frac{1}{\log N}$,  then $\xi_N$ converges to a Gumbel random variable. 
Corollary~\ref{thm:Gumbel} gives an extension of this result in the regime where $\beta N \sim \gamma$.
Finally, let us mention that in the regime where $\beta\to0$ and $\beta N\to+\infty$, using the tridiagonal random matrix representation of the Gaussian $\beta$--ensemble, Pakzad obtained in \cite[Theorem 1.1]{Pakzad19} a large deviations principle at speed $\beta N$ for the largest eigenvalue. 
This implies that in this regime, the largest eigenvalue to $\sqrt{2}$ in probability. 
However, we still expect to observe that the largest eigenvalues has Gumbel fluctuations around  $\sqrt{2}$ -- see the discussions in~\cite{AD14a}.  
This is in contrast with the regime where $\beta>0$ is fixed and the fluctuations are described in terms of the stochastic Airy operator  \cite{RRV11}.  
 
 \paragraph{Acknowledgments.}
G.L. is supported by the SNSF Ambizione grant S-71114-05-01. G.L. thanks Trinh Khanh Duy for interesting discussions about the problem studied in this article.

\section{Proof of Theorem~\ref{thm:LF}} \label{sect:LF}

Throughout this section, we assume that $\beta N\to \gamma$ as $N\to+\infty$ for $\gamma \ge 0$ and that the interaction kernel $\g$ and the potential $V$ satisfy the Assumptions~\ref{ass:2}.
Without loss of generality, we also assume that $\beta N\le \kappa$ for all $N\in\N$ for a fixed $\kappa>0$ and that the potential $\widetilde{V}_\kappa \ge 0$. 
Moreover, we rely only on the fact that the empirical measure $\widehat{\mu}_N$ converges in probability to $\mu_\gamma \in \M(\X)$. 

\medskip

The first step of the proof of Theorem \ref{thm:LF} consists in obtaining a (uniform) bound for the density of state  $\rho_N$.
As, we see in Section~\ref{sect:density}, this is a rather straightforward consequence of the Assumptions~\ref{ass:2}. 
Then, in Section~\ref{sect:correlation2}, we use this bound to deduce the convergence of the correlation functions of the local point process \eqref{def:Xi} and complete the proof of Theorem~\ref{thm:LF}. The argument to obtain this convergence is inspired from  \cite{BP15, NT}. 

\subsection{Estimates for the density of state} \label{sect:density}

We define the density of states (or first marginal of $\P_N$) by
\[
 \rho_N(u)  := \int \frac{e^{-\Ha_N(u,x_{2} ,\dots, x_N)}}{\Z_N^{(V)}} \omega(\d x_{2}) \cdots \omega(\d x_N) , \qquad u\in\X . 
\]
For any $N\in\N$, $\rho_N$ is a probability density function on $\X$. 

\medskip

First, let us observe that by Jensen's inequality, we immediately have a lower bound for the partition function \eqref{def:Z}. Indeed, by \eqref{F3}, the reference measure $\d\mu_0 = \mathrm{L}_0^{-1} e^{- V} \d\omega$ satisfies $F_\kappa(\mu_0) \ge \frac\kappa2 \En(\mu_0)$  and this implies that
\begin{equation*} \label{Z}
\begin{aligned}
\Z_N^{(V)}
& = \mathrm{L}_0^N \int e^{ - \beta \sum_{ i < j } \g(x_i-x_j) } \prod_{i=1}^N \mu_0(\d x_i) \\
%&=    \mathrm{L}_0^N \int e^{ - \frac\beta2  \sum_{ i \neq j } \g(x_i-x_j)  - \sum_i  \log \varrho(x_i)} \prod_{i=1}^N \mu_\kappa(\d x_i) \\
& \ge  \mathrm{L}_0^N \exp\bigg( - \beta \int   \sum_{ i < j } \g(x_i-x_j)   \prod_{i=1}^N \mu_0(\d x_i) \bigg)  \\
& =  \mathrm{L}_0^N \exp\big( - N\tfrac{\beta(N-1)}{2}  \En(\mu_0) \Big)
\end{aligned}
\end{equation*}
Since $F_\kappa(\mu_0) \ge 0$, this shows that if $\beta (N-1) \le \kappa$, 
\begin{equation} \label{ZLB}
\Z_N^{(V)} \ge  \mathrm{L}_0^N e^{- N \F_\kappa(\mu_0) } . 
\end{equation}
 
Moreover, if $f\in L^1(\nu_\kappa)$ is a non--negative function, then by \eqref{def:vartheta} and \eqref{def:mu0},  we obtain the trivial estimate
\begin{equation} \label{intest}
\E_N\bigg[ \int f  \d\emp_N \bigg] \le C(N,\beta) \bigg( \int f  \d\nu_\kappa\bigg)^N . 
\end{equation}

\begin{proposition}[Wegner estimate] \label{prop:Wegner}
Let  $\widetilde{V}(x) = V(x) - \beta(N-1) \vartheta(x)$ for $x\in\X$. 
For any $N\in\N$ and for all $u\in\X$,
\[
\rho_N(u)  \ll e^{-\widetilde{V}(u)} 
\]
where the implied constant depends only on $\kappa>0$. 
\end{proposition}

\begin{proof}
We can suppose that $N\ge 2$, otherwise the estimate is trivial. 
By definitions, we can rewrite
\begin{equation} \label{Zratio}
\Z_N^{(V)}   =  \Z_{N-1}^{(V)}  \mathrm{L}_0 \int  \E_{N-1} \big[  e^{-\beta (N-1) \int \g(u, \cdot) \d\emp_{N-1} }\big] \mu_0(\d u)
\end{equation}
and for all $u\in\X$, 
\begin{equation} \label{rho1}
\rho_N(u) =  \frac{  \Z_{N-1}^{(V)} }{ \Z_{N}^{(V)}}   \E_{N-1} \big[  e^{-\beta \sum_{i=1}^{N-1} \g(u, x_i)}\big] e^{-V(u)}  . 
\end{equation}
Let us check that the random variable $\int \g(u, \cdot) \d\emp_{N-1}$ is integrable with respect to the probability measure $\P_{N-1} \times \mu_0$. 
On the one--hand,  it follows from the Assumptions~\ref{ass:1}.3) that 
 \begin{align} 
\int \E_{N-1}\bigg[ \int \g_+(u, \cdot) \d\emp_{N-1} \bigg] \mu_0(\d u) 
&\notag \le 1 +  \iint \g^1(u,x) \mu_0(\d u) \rho_N(\d x) \\
&\notag \le 1+  \mathrm{C}_0^{-1} \bigg(  \int  \g^1(u,x)^p \omega(\d x) \bigg)^{1/p}   \bigg(  \int e^{-V(x)} \omega(\d x) \bigg)^{1/q} \\
& \label{estimate0} \ll 1
\end{align}
by H\"older's inequality and using that $V\ge 0$.  On the other hand by \eqref{def:vartheta} and \eqref{intest}, we have
\[ \begin{aligned} 
\int \E_{N-1}\bigg[ \int \g_-(u, \cdot) \d\emp_{N-1} \bigg] \mu_0(\d u) 
&\le  \int \vartheta(u) \mu_0(\d u)  +  \E_N\bigg[ \int \vartheta  \d\emp_N \bigg]   <+\infty 
\end{aligned} \]
where we used that $ \displaystyle \int \vartheta(u) (\d u)  <+\infty $ since $\widetilde{V}_\kappa$ satisfies the 
Assumptions~\ref{ass:1}.1) for any $\kappa \ge 0$. 
Hence, by applying Jensen's inequality, it follows from \eqref{Zratio} that 
\[ \begin{aligned}
\frac{ \Z_{N}^{(V)}}{ \Z_{N-1}^{(V)}} % & =  \int  \E_{N-1} \big[  e^{-2\beta (N-1) \int \g(x- \cdot) \d\emp_{N-1} }\big] \mu_0(\d x) \\
& \ge  \mathrm{L}_0 \exp\bigg( - \beta (N-1)  \E_{N-1}\bigg[ \int \U^{\mu_0} \d\emp_{N-1}\bigg]   \bigg)    .
\end{aligned}\] 
Note that we used the symmetry of the kernel $\g$. 
Moreover, it follows from the estimate \eqref{estimate0} that $\E_{N-1}\bigg[ \int \U^{\mu_0} \d\emp_{N-1}\bigg]  \le C$ for a universal constant $C>0$. 
This implies that if $\beta (N-1) \le \kappa$, then
$\frac{\Z_{N-1}^\beta}{\Z_{N}^\beta}   \ll 1$. 
By \eqref{rho1} and   \eqref{def:vartheta}, this shows that the density of states satisfies
\begin{equation} \label{UB1}
\begin{aligned}
\rho_N(u)  & \ll   \E_{N-1} \big[  e^{-\beta \sum_{j=1}^{N-1} \g(u ,x_j)}\big] e^{-V(u)}   \\
& \le    \E_{N-1} \big[  e^{\beta \sum_{j=1}^{N-1} \vartheta(x_j)}\big] e^{-\widetilde{V}(u)}    , 
\end{aligned}
\end{equation}
where $\widetilde{V}(u) = V(u) - \beta(N-1) \vartheta(u)$.
Finally, by  Jensen's inequality again,  it holds for any $N\ge r \ge1$, 
\[
 \E_{N} \big[  e^{\beta r \sum_{j=1}^{N} \vartheta(x_j)}\big]
\le  \Big(   \E_{N} \big[  e^{\kappa \sum_{j=1}^{N} \vartheta(x_j)}\big] \big)^{\frac{r}{N}} 
=  \bigg( \frac{ \Z_N^{(V- \kappa \vartheta)} }{ \Z_{N}^{(V)}} \bigg)^{\frac{r}{N}} , 
\]
where we used that the function $\vartheta \ge 0$ and $\beta N \le \kappa$.
Using the upper--bound \eqref{ZUB} (applied to the potential $\widetilde{V}_\kappa=V- \kappa \vartheta$ which satisfies the Assumptions~\ref{ass:1}.1)) as well as the lower--bound \eqref{ZLB}, this shows that  for any $N\ge r \ge1$, 
\begin{equation}\label{UB2}
 \E_{N} \big[  e^{\beta r \sum_{i=1}^{N} \vartheta( x_i)}\big] \le C(\kappa)^r .
\end{equation}
%where the constant $C>0$ does not depend on $r$. 
%  \E_{N} \big[  e^{-\beta \sum_{i=1}^{N} \g(u- x_i)}\big]\le  e^{\kappa \vartheta(u)}  
By combining the estimates \eqref{UB1} and \eqref{UB2} with $r=1$, this completes the proof.
\end{proof}

\begin{remark} \label{rk:integrability}
If the empirical measure $\widehat{\mu}_N\Rightarrow\mu_\gamma$ in probability as $\beta N\to \gamma$ and $N\to+\infty$, then $\rho_N \Rightarrow\mu_\gamma$ in the same regime.
Thus, if  $\mu_\gamma$ is absolutely continuous,  it can be inferred from the uniform bound of Proposition~\ref{prop:Wegner} and Lebesgue differentiation theorem that  its density $\mu_\gamma \ll \nu_\kappa$. 
\end{remark}

Let us also record the following consequence of the Law of large numbers and Proposition~\ref{prop:Wegner}. 

\begin{proposition} \label{lem:lsc}
If $f :\X \mapsto (-\infty,+\infty]$ is a lower semicontinuous function which is bounded from below, then as $\beta N \to\gamma$, 
\[
 \liminf_{N\to+\infty} \E_N\bigg[\int f \d\emp_N  \bigg]  \ge \int f \d\mu_\gamma . 
\]
Moreover, for any  lower semicontinuous function $f\in L^1(\nu_\kappa)$, we have as $\beta N \to\gamma$, 

\[
 \lim_{N\to+\infty} \E_N\bigg[ \bigg| \int f \d\emp_N  -  \int f \d\mu_\gamma \bigg| \bigg] =0 . 
\]
\end{proposition}

\begin{proof}
Without loss of generality, we can assume that $f \ge 0$. 
By Skorokhod's representation Theorem and Proposition~\ref{prop:eq}, there exists a sequence of random measures $\nu_N$ taking values in $\M(\R^d)$ with the same law as $\widehat{\mu}_{N}$ such that $\nu_N \Rightarrow \mu_\gamma$ almost surely as $N\to+\infty$ (for the topology of weak convergence). 
By Portmanteau's theorem, since the function $f$ is lower--semicontinuous, almost surely as $N\to+\infty$,
\begin{equation}  \label{liminf}
 \liminf_{N\to+\infty}\int f \d\nu_N   \ge \int f \d\mu_\gamma
\end{equation}
Then, by by Fatou's Lemma, this implies that 
\[
 \liminf_{N\to+\infty} \E\bigg[\int f \d\nu_N  \bigg]  \ge \int f \d\mu_\gamma . 
\]
This proves the first claim (the RHS is allowed to be $+\infty$). 
For the second claim, we can still assume that $f\ge 0$ and observe that by Remark~\ref{rk:integrability}, $f\in L^1(\nu_\kappa)$. 
Then, for a given small $\epsilon>0$, observe that \eqref{liminf} implies that if  $N$ is sufficiently large (depending only on $\epsilon>0$), then 
\begin{equation} \label{estimate3}
\E \bigg[ \bigg| \int f \d\nu_N   -    \int f \d\mu_\gamma  \bigg| \bigg] \le  \epsilon  + 
\E \bigg[ \int f \d\nu_N    \bigg] -   \int f \d\mu_\gamma  .  
\end{equation}
Moreover, since $f$ is lower semicontinuous, there exists a sequence of continuous functions $0\le f_k \le k$ such that $f_k \nearrow f$ pointwise. By \eqref{LLN}, it holds for any $k\in\N$, 
\begin{equation} \label{estimate1}
\lim_{N\to+\infty}\E \bigg[ \int f_k \d\nu_N   \bigg]  =  \int f_k \d \mu_\gamma 
\end{equation}
and by Proposition~\ref{prop:Wegner} as $ \beta N \le \kappa$, 
\begin{equation} \label{estimate2}
\E \bigg[ \int (f-f_k) \d\nu_N   \bigg]  = \int  (f(x)-f_k(x)) \rho_N(x) \omega(\d x) \ll  \int  (f(x)-f_k(x))  \nu_\kappa(\d x) .
\end{equation}
Since $f \in L^1(\nu_\kappa)$, the RHS of \eqref{estimate2} converges to 0 as $k\to+\infty$ and, by \eqref{estimate1},  this implies that
\[
\limsup_{N\to+\infty}\E \bigg[ \int f \d\nu_N    \bigg]  \le  \int f \d\mu_\gamma    
\]
By \eqref{estimate3} and since $\epsilon>0$ is arbitrary, this proves the second claim.  
\end{proof}

\subsection{Convergence of the correlation functions} \label{sect:correlation2}

Let us define the random function
\begin{equation} \label{def:X}
\mathrm{X}_N(u) : = \int  \g(u,\cdot)  \d\widehat{\mu}_{N} = \frac 1N {\textstyle \sum_{j=1}^N}  \g(u,x_j) , \qquad u\in \X .
\end{equation}
This function is the potential generated by  the empirical measure $\widehat{\mu}_{N}$ and it is lower semicontinuous. 
Let us denote the marginals of the probability measure $\P_N$ by $\big(\rho_N^{(k)}\big)_{k=1}^N$ . That is for any integer $k < N$ and  $\u=(u_1,\dots,u_k) \in \X^k$,
\begin{equation}\label{def:rho}
 \rho_N^{(k)}(\u)   = \int \frac{e^{-\Ha_N(u_1,\dots, u_k ,x_{k+1} ,\dots, x_N)}}{ \Z_{N}^{(V)}} \omega(\d x_{k+1}) \cdots \omega(\d x_N) . 
\end{equation}
Then $\rho_N^{(1)}$ corresponds to the density of states and let us observe that according to \eqref{def:X}, we can rewrite for any $k=1,\dots, N$, 
\begin{equation} \label{marginals}
 \rho_N^{(k)}(\u) =  \frac{ \Z_{N-k}^{(V)}}{ \Z_{N}^{(V)}}  
 e^{-\Ha_k(\u)}  \E_{N-k} \big[  e^{-\beta (N-k) \sum_{i=1}^k \mathrm{X}_{N-k}(u_i)}\big] ,
 \qquad \u\in \X^k . 
\end{equation}

Fix $E\in\S_V$ and a normal neighborhood $U$ of $E$. 
It follows from formula \eqref{R:asymp} in the Appendix~\ref{sect:pp} that the correlation functions of the local point process \eqref{def:Xi} satisfies for any fixed $k\in\N$ as $N\to+\infty$
\begin{equation} \label{Rk}
 R_N^{(k)}(\x)= \rho_N^{(k)}\Big|_U\big(\varphi^{-1}(x_1N^{-1/n}), \cdots  ,  \varphi^{-1}(x_kN^{-1/n})\big)
\big( 1+\o(1)  \big)  , 
\end{equation}
uniformly for all $\x = (x_1,\dots ,x_n)$ in compact subsets of $(\R^n)^k$ . 
Our goal in this section is to establish that for any $k \ge 1$, $R_N^{(k)} \to \mu_\gamma(E)^k$ as $N\to+\infty$ and $\beta N \to\gamma$ for almost all $\x \in (\R^n)^k$ and to deduce  Theorem~\ref{thm:LF}. 
Our first and main Lemma deals with the asymptotics of the random potential $\mathrm{X}_N$.

\begin{lemma} \label{lem:main}
Fix $u \in\X$ and let $u_N\in\X$ be any sequence such that $u_N\to u$ as $N\to+\infty$. We have as  $\beta N \to\gamma$, 
\[
\lim_{N\to+\infty} \E_N \Big[ \big| \mathrm{X}_N(u_N) -  \U^{\mu_\gamma}(u) \big| \Big]  =0 , 
\]
where $ \displaystyle \U^{\mu_\gamma}(u) = \int \g(u,x) \mu_\gamma(\d x) $ is the equilibrium potential. 
\end{lemma}

\begin{proof}
We claim that for any fixed $u\in\X$, the function $x\mapsto\g(u,x)$ lies in $L^1(\nu_\kappa)$.
Indeed, under the Assumptions~\ref{ass:1}.3), by H\"older's inequality,  it holds for any $k \in \N$,  
\begin{equation}  \label{estimate4}
\begin{aligned}
\int \g^k(u,x) \nu_\kappa(\d x) 
& \ll \int \g^k(u,x)  e^{-\widetilde{V}_\kappa(x)}  \omega( \d x) \\
& \le  \bigg(\int \g^k(u,x)^p\, \omega(\d x) \bigg)^{1/p}  \bigg( \int  e^{-\widetilde{V}_\kappa(x)} \omega(\d x) \bigg)^{1/q}  \ll \mathrm{c}_k 
\end{aligned} 
\end{equation}
where we also used that $\widetilde{V}_\gamma \ge 0$ and that $\widetilde{V}_\kappa$ satisfies the Assumptions~\ref{ass:1}.1). 
In addition, since $|\g_k(u,x)| \le k + \vartheta(u) + \vartheta(x)$ by \eqref{def:vartheta}, it  holds that for all $k \in \N$, 
\begin{align}
\int \big| \g_k(u,x) \big|  \nu_\kappa(\d x) 
  &\notag \ll  k +  \int \vartheta(x) e^{-\widetilde{V}_\kappa(x)}  \omega(\d x) \\
& \label{estimate5}
\ll  k \int e^{-\widehat{V}_{\kappa + 1/k}(x) }  \omega(\d x)  <+\infty
\end{align}
where we used that $k +  \vartheta \le ke^{\vartheta/k}$ at the second step. 
Hence, since $\g$ is lower continuous, by Proposition~\ref{lem:lsc}, we obtain that for any fixed $u\in\X$,  as  $\beta N \to\gamma$,  
\[
\lim_{N\to+\infty} \E_N \Big[ \big| \mathrm{X}_N(u) -  \U^{\mu_\gamma}(u) \big| \Big]  =0 . 
\]
So it remains to show that if $u_N\to u$ in $\X$ , then
\begin{equation} \label{estimate7}
\lim_{N\to+\infty} \E_N \Big[ \big| \mathrm{X}_N(u_N) -   \mathrm{X}_N(u) \big| \Big]  =0 . 
\end{equation}
By the triangle inequality, we have 
\[
 \E_N \Big[ \big| \mathrm{X}_N(u_N) -   \mathrm{X}_N(u) \big| \Big]  \le \E_N \bigg[  \int \big| \g(u_N, \cdot ) - \g(u,\cdot) \big| \d\widehat{\mu}_{N} \bigg]  = \int \big| \g(u_N, x) - \g(u, x) \big|  \rho_N(x) \omega(\d x) . 
\]
By Proposition~\ref{prop:Wegner}, as $\beta N \le\kappa$, this implies that 
\begin{equation}  \label{estimate6}
 \E_N \Big[ \big| \mathrm{X}_N(u_N) -   \mathrm{X}_N(u) \big| \Big]  \ll \int   \big| \g(u_N, x) - \g(u, x) \big| \nu_\kappa(\d x) . 
 \end{equation}
 
 Observe that since $\vartheta$ is continuous, the implied constant in the estimate \eqref{estimate5} are uniform for all $u$ in a compact set of $\X$. Hence, since $\g_k$ is continuous, by the dominated convergence Theorem,    it holds for all $k\in\N$, 
\[
\lim_{N\to+\infty} \int \big| \g_k(u_N, x) - \g_k(u,x) \big| \nu_\kappa(\d x) =0  .
\]
Moreover, by \eqref{estimate4}, it holds uniformly for all $u\in \X$ and for all $k\in\N$,
\[
\int  \g^k(u, x)  \nu_\kappa(\d x)   \ll \mathrm{c}_k .
 \]
By combining these estimates with \eqref{estimate6}, we obtain that for all  $k\in\N$, 
\[
\limsup_{N\to+\infty} \E_N \Big[ \big| \mathrm{X}_N(u_N) -   \mathrm{X}_N(u) \big| \Big]  \ll  \mathrm{c}_k  . 
\]
By assumptions, since $\mathrm{c}_k \to0$ as $k\to+\infty$ this proves \eqref{estimate7}.
\end{proof}

From Lemma \ref{lem:main}, we can deduce the asymptotics of ratios of partition functions.

\begin{corollary} \label{cor:ratio}
Fix $k \in\N$ and $\u\in \X^k$. For any $ i \in\{1,\dots,k\}$, let $u_{N,i} $ be a sequence such that  $u_{N,i} \to u_i $ in $\X$ as $N\to+\infty$. We have as $\beta N \to\gamma$, 
\[
\lim_{N\to+\infty}  \E_{N} \big[  e^{-\beta N  \sum_{i=1}^k  \mathrm{X}_N(u_{N,i})}\big] = e^{-\gamma \sum_{i=1}^k  \U^{\mu_\gamma}(u_i)} . 
\]
Moreover, it holds for any fixed $k \in\N$,
 \[
\lim_{N\to+\infty} \tfrac{ \Z_{N-k}^{(V)}}{ \Z_{N}^{(V)}}   =  \mathrm{L}_\gamma^{-k} ,
\]
where the constant $\mathrm{L}_\gamma>0$ is as in  \eqref{eq:mu}.
\end{corollary}

\begin{proof}
It follows immediately from  Lemma~\ref{lem:main} that under our assumptions, as $N\to+\infty$, 
\begin{equation} \label{probacvg}
\beta N {\textstyle \sum_{i=1}^k } \mathrm{X}_N(u_{N,i})  \to  \gamma   {\textstyle \sum_{i=1}^k }  \U^{\mu_\gamma}(u_i) 
\end{equation}
 in probability (with respect to $\P_N$). 
 Moreover, by \eqref{def:vartheta} and using the estimate \eqref{UB2}, it holds for any $r\ge 1$,
\begin{equation} \label{UB3}
\begin{aligned}
   \E_{N} \big[  e^{- r\beta N  \sum_{i=1}^k  \mathrm{X}_N(u_{N,i})}\big] 
& \le   e^{r   \kappa\sum_{i=1}^k  \vartheta(u_{N,i})}
   \E_{N} \big[  e^{ r k \beta \sum_{j=1}^N \vartheta( x_j) }\big]   \\
   & \ll C^{rk} e^{r   \kappa\sum_{i=1}^k  \vartheta(u_{N,i})} ,
\end{aligned}
\end{equation}
where the constant $C>0$ does not depend on $k \in \N$ and $r\ge 1$. 
Since $\vartheta$ is continuous on $\X$, this shows that the random variables $e^{- \beta N  \sum_{i=1}^k  \mathrm{X}_N(u_{N,i})}$ lie in $L^r(\P_N)$ for any $r\ge 1$, $N \ge r k$ and  $\u\in \X^k$.
Consequently, we deduce from \eqref{probacvg} that for any fixed $k\in\N$ and  $\u\in \X^k$,
\begin{equation} \label{limit1}
\lim_{N\to+\infty}     \E_{N} \big[  e^{- \beta N  \sum_{i=1}^k \mathrm{X}_N(u_{N,i})}\big] 
= e^{-\gamma \sum_{i=1}^k  \U^{\mu_\gamma}(u_i)} .
\end{equation}
For the second claim, let us observe that by integrating formula \eqref{marginals}, we obtain for any $k\in\{1,\dots N-1\}$, 
\[
 \frac{ \Z_{N}^{(V)}}{ \Z_{N-k}^{(V)}}  
= \int   \E_{N-k} \big[  e^{-\beta (N-k) \sum_{i=1}^k \mathrm{X}_{N-k}(u_i)}\big] e^{-\Ha_k(\u)}  \omega(\d u_1) \cdots \omega(\d  u_k)  . 
\]
From \eqref{H2}, we have $\Ha_k(\u) \ge {\textstyle \sum_{i=1}^k}  \widetilde{V}_\kappa(u_i)$  and using \eqref{UB3} with $r=1$ and $u_{N,i}=u_i$, it holds for any  $\u\in\X^n$,
\begin{equation} \label{est4}
  \E_{N-k} \big[  e^{-\beta (N-k) \sum_{i=1}^k \mathrm{X}_{N-k}(u_k)}\big] e^{-\Ha_k(\u)}  \ll C^k  e^{- \sum_{i=1}^k  \widetilde{V}_{\kappa}(u_i)} . 
\end{equation}

By Assumptions~\ref{ass:2}, the RHS of \eqref{est4} is an integrable function on $\X^k$. Since for almost every  $\u\in\X^k$, $e^{-\Ha_k(\u)}   \to  e^{-\sum_{i=1}^kV(u_i)} $ as $\beta\to 0$, it follows from \eqref{limit1}
%\footnote{The sam limit holds along the sequence $N-n$.} 
and the dominated convergence theorem that for any fixed $k \in\N$, as $\beta N \to \gamma$, 
\[
\lim_{N\to+\infty}  \frac{ \Z_{N}^{(V)}}{ \Z_{N-k}^{(V)}}     = \int e^{- \sum_{i=1}^k \big( \gamma\U^{\mu_\gamma}(u_i)+ V(u_i) \big)}  \omega(\d u_1) \cdots \omega(\d  u_k)  . 
\]
By \eqref{eq:mu}, since $\mu_\gamma$ is a probability measure, this completes the proof.
\end{proof}

\begin{remark} \label{rk:equilibrium}
It follows from formula \eqref{rho1} and Corollary~\ref{cor:ratio} with $k=1$, that the density of states satisfies for all $u\in\X$, as $\beta N \to\gamma$,
\[
\lim_{N\to+\infty} \rho_N(u) =  \mathrm{L}_\gamma^{-1} e^{- \gamma \U^{\mu_\gamma}(u) - V(u)}  . 
\]
Since we also know that  $\rho_N \Rightarrow\mu_\gamma$ in the same regime, this gives an alternative proof of the self--consistent equation \eqref{eq:mu}  which is satisfied by the equilibrium density $\mu_\gamma$.  
\end{remark}

We are now ready to complete the proof of our main result, Theorem~\ref{thm:LF}.

\begin{proof}[Proof of Theorem~\ref{thm:LF}]
Recall that $U$ is a normal neighborhood of $E \in \S_V$  and that $\varphi$ is the normal coordinate chart in $U$. In particular $\varphi(U)=  \{x\in \R^n : |x|\le \delta \} $ for a $\delta>0$. 
Let us fix a large ball $\mathcal{K} \subset \R^n$. 
According to \eqref{marginals} and \eqref{Rk}, if $N$ is sufficiently large (depending on $\mathcal{K}$), the correlation functions of the local process $\Xi_N$ satisfy for any fixed $k\in\N$ and uniformly for all $\x \in \mathcal{K}^k$, 
\begin{equation} \label{R1}
 R_N^{(k)}(\x) =  \tfrac{ \Z_{N-k}^{(V)}}{ \Z_{N}^{(V)}}   
  \E_{N-k} \big[  e^{-\beta (N-k) \sum_{i=1}^k \mathrm{X}_{N-k}(u_{N,i})}\big]    e^{-\Ha_k(u_{N,1}, \dots, u_{N,k})}\big( 1+\o(1)  \big)  , 
\end{equation}
where $u_{N,i} = \varphi^{-1}(x_i/N^{1/n})$ for $i\in\{1,\dots,k\}$. 
Notice that $u_{N,i}\to E$ for all $i\in\{1,\dots,k\}$ as $N\to +\infty$, so that by Corollary~\ref{cor:ratio}, as $\beta N \to\gamma$, 
\begin{equation} \label{lim1}
\lim_{N\to+\infty}  \E_{N} \big[  e^{-\beta (N-k)  \sum_{i=1}^n  \mathrm{X}_{N-k}(u_{N,i})}\big] = e^{-\gamma k \U^{\mu_\gamma}(E)} . 
\end{equation}
Moreover,  it follows from the Assumptions~\ref{ass:1}.2) that for any $\x \in \mathcal{K}^k$ with $x_1 \neq \cdots \neq x_k$, 
\begin{equation} \label{lim2}
\lim_{N\to+\infty}\Ha_k(u_{N,1},\dots, u_{N,k})  = kV(E) , 
\end{equation}
where we used that by definition of the normal coordinates, $ \dist(u_{N,i}, u_{N,j}) = |x_i-x_j| N^{-1/n} $, so that
\begin{equation} \label{estimate8}
 \beta\, |\g(u_{N,i}, u_{N,j})|  \le \kappa  \frac{\dist(u_{N,i}, u_{N,j})^n  | \g(u_{N,i}, u_{N,j}) |}{ |x_i-x_j|^n}  
\end{equation}
and the RHS of \eqref{estimate8} converges to 0 as $N\to+\infty$. Note that we also used the continuity of the potential $V$ to obtain \eqref{lim2}. 
By combining \eqref{lim1}, \eqref{lim2} with \eqref{R1} and using the second asymptotics from  Corollary~\ref{cor:ratio},  we conclude that  for any fixed $k\in\N$ and for almost all $\x \in \mathcal{K}^k$, 
 \begin{equation} \label{lim3}
 \begin{aligned}
\lim_{N\to+\infty}   R_N^{(k)}(\x)  &=    \mathrm{L}_\gamma^{-k}  e^{- k\big( \gamma  \U^{\mu_\gamma}(E) + V(E) \big)} \\
& = \mu_\gamma(E)^k . 
\end{aligned}
\end{equation}
The second step follows from the equation \eqref{eq:mu} for the equilibrium density. 

\medskip

Finally, from \eqref{R1} and using the estimate \eqref{est4} as well as the fact that $\widetilde{V}_{\kappa} \ge 0$, we obtain that for all $\x \in \mathcal{K}^k$, 
\[
R_N^{(k)}(x_1,\dots, x_k) \le C^k \frac{\Z_{N-k}^\beta}{\Z_{N}^\beta} , 
\]
where the constant $C>0$ does not depend only on $k\in\N$. 
By Corollary~\ref{cor:ratio},  this shows that uniformly for all  $\x \in \mathcal{K}^k$, 
\begin{equation} \label{estimate9}
R_N^{(k)}(x_1,\dots, x_k) \le C^{2k} . 
\end{equation}
According to Lemma \ref{lem:ppwcvg}, if we combine the limits \eqref{lim3} with the estimates \eqref{estimate9}, we have proved that  the local process $\Xi_N$ converges in distribution as $N\to+\infty$  to a homogeneous Poisson point process on $\R^n$ with intensity $\mu_\gamma(E)$.
\end{proof}

\section{Local fluctuations near a \emph{boundary point}}  \label{sect:edge}

The goal of this section is to prove Theorem~\ref{thm:REF} and Theorem~\ref{thm:LEF}. 
Like in Section~\ref{sect:LF},  we assume that $\beta N\to\gamma$ as $N\to+\infty$ and that $\beta N\le \kappa$ for all $N\in\N$ for a fixed $\kappa>0$. 
We work under the following general conditions. 

 \begin{assumption} \label{ass:3}
Let $\X= \R^n$ and $\gamma \ge 0$. Let us suppose that the interaction kernel $\g \ge 0$ satisfies the  Assumptions~\ref{ass:1}.2)3) and that 
 $\g_k(x,u) \to 0$ as $x\to +\infty$ for all $u\in\R^n$ and $k\in\N$.  
In addition, suppose that $V:\R^n \to [0,+\infty)$ is continuous, $C^2$ outside of a compact set, that  $e^{-V}\in L^1(\R^n)$ and let $E_N \in \R^n$ be a diverging sequence which satisfies the following conditions:  $\nabla V(E_N) = \alpha_N \upsilon_N$ with $\alpha_N>0$, $\|\upsilon_N\| =1$  and it holds  as $N\to+\infty$, 
\begin{equation} \label{c1}
\frac{Ne^{-V(E_N)}}{ \mathrm{L}_\gamma \alpha_N^n} \to 1 , 
\qquad \upsilon_N\to \upsilon  ,
\qquad\text{and}\qquad  \alpha_N |E_N| \to +\infty .
\end{equation}
Finally, let $\varphi_N(x) = E_N + \alpha_N^{-1}\psi(x) $ where $\psi:\R^n \to \R^n$ be a diffeomorphism with Jacobian 
$\mathscr{J}(\psi)=1$ and suppose that for any compact set $\mathscr{K} \subset \R^n$, as $N\to+\infty$,
\begin{equation} \label{c2}
% \inf_{x \in \mathscr{K} }\left\| \varphi_N(x) \right\| \to +\infty
 %\qquad\text{and}\qquad
\alpha_N^{-2} \sup_{x \in \mathscr{K} }\left\| \nabla^2V\big(\varphi_N(x)\big) \right\| \to 0 . 
\end{equation}
 \end{assumption}
 
 Let us point out that we assume that the interaction kernel $\g$ decays at $\infty$ and that   these assumptions are very similar to those of Lemma~\ref{lem:Poissonedge} which deals with the case of i.i.d. particles. They are slightly technical but they apply to a large class of potentials $V\in C^2(\R^n)$, even though constructing the sequence $E_N$ could be difficult. 
By adapting the method of the proof of Theorem~\ref{thm:LF} presented in Section~\ref{sect:LF}, we obtain the following result.

 \begin{theorem} \label{thm:limitcase}
 Suppose that the interaction kernel and the potential $V: \R^n \to [0,+\infty)$ satisfies the Assumptions~\ref{ass:3}.
Then, the point process $\Xi_N = \sum_{j=1}^N  \boldsymbol\delta_{ \varphi_N^{-1}(x_j)}$  obtained by zooming around the point~$E_N$  converges in distribution  as $\beta N \to\gamma$ and $N\to+\infty$ to a Poisson point process with intensity $\theta(x) = e^{-\upsilon \cdot \psi(x)}$ on $\R^n$. 
  \end{theorem}
 
In order to prove Theorem~\ref{thm:limitcase}, we relie on the following Lemma.
 
 \begin{lemma} \label{lem:lim0}
 Under the Assumptions~\ref{ass:3}, for any $x\in\R^n$, $\mathrm{X}_N(\varphi_N(x)) \to 0$ in $L^1(\P_N)$. 
 \end{lemma}
 
\begin{proof}
By \eqref{def:X}, the random function $\mathrm{X}_N \ge 0$ on $\R^n$, so it suffices to show that 
$\E_N\big[\mathrm{X}_N(\varphi_N) \big] \to 0$ as $N\to+\infty$ where $\varphi_N = \varphi_N(x)$ ($x\in\R^n$ is fixed).
 By  Proposition~\ref{prop:Wegner},  we have
\[\begin{aligned}
\E_N\big[\mathrm{X}_N(\varphi_N)\big]  &=  \int \g(\varphi_N,u) \rho_N(\d u) \\
& \ll   \int \g(\varphi_N,u) e^{-V(u)} \d u . 
\end{aligned}\]
Fix a $k\in\N$. Since $0\le \g_k \le k$, and $\g_k(\varphi_N,u) \to 0$ as $N\to+\infty$ for all $u\in\R^n$ (in particular, the last condition in \eqref{c1} implies that $|\varphi_N| \sim |E_N|$ which diverges as $N\to+\infty$), we deduce from Lebesgue's dominated convergence theorem that
\[
\lim_{N\to+\infty} \int \g_k(\varphi_N,u)  e^{-V(u)} \d u  =0  . 
\]
Hence, it follows from the Assumptions~\ref{ass:1}.3) and H\"older's inequality that 
\[
\limsup_{N\to+\infty} \E_N\big[\mathrm{X}_N(\varphi_N)\big]   \ll \mathrm{c}_k . 
\]
Since $ \mathrm{c}_k \to 0$ as $k\to+\infty$, this completes the proof.
\end{proof}

\begin{proof}[Proof of Theorem~\ref{thm:limitcase}]
According to formulae \eqref{marginals}, \eqref{correlations} and using the change of variables \eqref{trans}, 
 the correlation functions of the point process $\Xi_N = \sum_{j=1}^N  \boldsymbol\delta_{ \varphi_N^{-1}(x_j)}$ are given by for all fixed $k\in\N$  and $\x \in (\R^n)^k$, 
\begin{equation} \label{R2}
 R_N^{(k)}(\x) =  \tfrac{ \Z_{N-k}^{(V)}}{ \Z_{N}^{(V)}}   
  \E_{N-k} \big[  e^{-\beta (N-k) \sum_{i=1}^k \mathrm{X}_{N-k}( \varphi_N(x_i))}\big]    e^{-\Ha_k( \varphi_N(x_1), \dots,  \varphi_N(x_k))}   \tfrac{N! \alpha_N^{-kn}}{(N-k)!}  
\end{equation}
where we used that the Jacobian of the map $\varphi_N$ equals to $\alpha_N^{-n}$ on $\R^n$. 
Let us fix $k\in\N$  and a compact set $\mathcal{K} \subset \R^n$. 

\medskip

First, according to Lemma \ref{lem:lim0} and since the random function $\mathrm{X}_{N-k} \ge 0$, we obtain that as $\beta N \to\gamma$, 
\begin{equation} \label{limit5}
  \E_{N-k} \big[  e^{-\beta (N-k) \sum_{i=1}^k \mathrm{X}_{N-k}( \varphi_N(x_i))}\big]     \to 1. 
\end{equation}
Second, note that the first condition in \eqref{c1} implies that $\alpha_N^{n} \ll N$. 
So, using the Assumptions~\ref{ass:1}.2), we  have  for $x_i\neq x_j$,  $i,j\in\{1,\dots,k\}$, as  $\beta N \le \kappa$,
\[ \begin{aligned}
\beta\, \g( \varphi_N(x_i),  \varphi_N(x_j))  & \le  \kappa \epsilon N^{-1} |\varphi_N(x_i)- \varphi_N(x_j)|^{-n} \\
& =   \kappa \epsilon \alpha_N^{n} N^{-1} |\psi(x_i)-\psi(x_j)| \\
& \ll \epsilon |\psi(x_i)-\psi(x_j)| . 
\end{aligned}\]
Since this holds for arbitrary small $\epsilon>0$, we obtain that as $N\to+\infty$,
\[ 
\beta\, \g( \varphi_N(x_i),  \varphi_N(x_j))  \to 0. 
\]
Moreover, we verify by a Taylor expansion that the conditions \eqref{c1}--\eqref{c2} from the Assumptions~\ref{ass:3} imply that as $N \to\infty$,
\begin{equation} \label{limit7}
\frac{N e^{-V(\varphi_{N}(x))}}{ \mathrm{L}_\gamma \alpha_N^n}  \to \theta(x) = e^{-\upsilon \cdot \psi(x)} ,
\end{equation}
where the convergence is uniform for $x \in \mathcal{K}$. 
Then, it follows that for almost all $\x  \in \mathcal{K}^k$, 
\begin{equation} \label{limit6}
\lim_{N\to+\infty} e^{-\Ha_k(u_{N,1}, \dots, u_{N,k})}   \tfrac{N! \alpha_N^{-kn}}{(N-k)!}  
=  \lim_{N\to+\infty}  \prod_{i=1}^k \frac{N e^{-V(u_{N,i})}}{ \alpha_N^n}   =  \mathrm{L}_\gamma^k\, {\textstyle   \prod_{i=1}^k} \theta(x_i) . 
\end{equation}
By combining \eqref{limit5}, \eqref{limit6} with Corollary~\ref{cor:ratio} in formula \eqref{R2}, we obtain that  for almost every $\x  \in \mathcal{K}^k$,  as $\beta N \to\gamma$ and $N\to+\infty$, 
\begin{equation} \label{limit8}
 R_N^{(k)}(\x)  \to {\textstyle   \prod_{i=1}^k} \theta(x_i) .  
 \end{equation}
Finally, since the convergence \eqref{limit7} is uniform, $\mathrm{X}_{N-k} \ge 0$ and $\g\ge 0$, we deduce  from   \eqref{R2} that there exists a constant $C(\kappa)>0$ (which depends only on the parameter $\kappa$) such that for all $\x  \in \mathcal{K}^k$,
\[
 R_N^{(k)}(\x)   \le  C(\kappa)^k  \tfrac{ \Z_{N-k}^{(V)}}{ \Z_{N}^{(V)}}  \le C(\kappa)^{2k}.    
\]
According to Lemma \ref{lem:ppwcvg}, the previous estimate and \eqref{limit8} show that the point process $\Xi_N = \sum_{j=1}^N  \boldsymbol\delta_{ \varphi_N^{-1}(x_j)}$ converges in distribution as $N\to+\infty$  to a (non--homogeneous) Poisson point process on $\R^n$ with intensity $\theta$.
\end{proof}

We can apply Theorem~\ref{thm:limitcase} to the  Riesz kernels \eqref{Riesz} to deduce the asymptotics from Theorem~\ref{thm:REF}. 

 \begin{proof}[Proof of Theorem~\ref{thm:REF}]
It suffices to verify that the potential $V(x) = |x|^\alpha$ and the sequence $E_N = \eta_N\upsilon$  satisfy the Assumptions~\ref{ass:3}. We have $\alpha_N = \alpha \eta_N^{\alpha-1})$ and by a Taylor expansion, we verify that $ \frac{Ne^{ -\eta_N^\alpha}}{ \mathrm{L}_\gamma \alpha_N^n} \to 1$ as $N\to+\infty$, 
so that the conditions \eqref{c1} holds. 
Moreover, the Hilbert--Schmidt norm of the matrix $\nabla^2V$ satisfies as $\eta_N \to+\infty$,
\[\begin{aligned}
\left\| \nabla^2V\big(\varphi_N(x)\big) \right\|
& = \alpha\sqrt{(\alpha-1)^2+n-1} \big|\varphi_N(x)\big|^{\alpha-2} , \qquad x\in \R^n \\
& =   \alpha\sqrt{(\alpha-1)^2+n-1}  \eta_N^{\alpha-2} \big(1+ \O(\eta_N^{-\alpha}) \big) , \qquad x\in\mathscr{K} , 
\end{aligned}\]
where the last error term is uniform. This implies that
\[
\alpha_N^{-2} \sup_{x \in \mathscr{K} }\left\| \nabla^2V\big(\varphi_N(x)\big) \right\| 
\ll \eta_N^{-\alpha} .
\]
Upon observing that by construction, $\upsilon \cdot \psi(x) = e_1\cdot x$ for all $x\in\R^n$, this  completes the proof.
 \end{proof}

Theorem~\ref{thm:limitcase} can only be applied to an interaction kernel $\g \ge 0$ which decays away from the diagonal. 
Our next result applies specifically to the log kernel \eqref{logkernel}. Let $n=1 \text{ or } 2$ and for $u,x\in \R^n$,
\[
\g(x,u) = \log |x-u|^{-1} \, , \qquad \vartheta(u) = \log(1+|u|) . 
\]
We also let $V: \R^n \to \R$ be a $C^2$ potential which satisfies the conditions \eqref{c3}.

 \begin{theorem} \label{thm:logedge}
  Let $E_N \in \R^n$ be a diverging sequence which satisfies the following conditions:
    $\nabla V(E_N) = \alpha_N \upsilon_N$ with $\alpha_N>0$,  $\|\upsilon_N\| =1$  and it holds  as $N\to+\infty$, 
\begin{equation} \label{c4}
\frac{Ne^{-V(E_N)+ \beta N \log|E_N|}}{ \mathrm{L}_\gamma \alpha_N^{n}} \to 1 , 
\qquad \upsilon_N\to \upsilon  , 
\qquad N^{-1} \log(\alpha_N) \to 0 
\quad\text{and}\quad   \alpha_N |E_N| \to +\infty .
\end{equation}
Let $\varphi_N(x) = E_N + \alpha_N^{-1}\psi(x) $ where $\psi:\R^n \to \R^n$ is a a diffeomorphism with Jacobian $\mathscr{J}(\psi)=1$ and suppose that the condition \eqref{c2} holds. 
Then, the point process $\Xi_N = \sum_{j=1}^N  \boldsymbol\delta_{ \varphi_N^{-1}(x_j)}$  converges in distribution  as $\beta N \to\gamma$ to a Poisson point process with intensity $\theta(x) = e^{-\upsilon \cdot \psi(x)}$ on $\R^n$. 
  \end{theorem}

The only substantial difference in the proof of Theorem~\ref{thm:logedge} compared with that of Theorem~\ref{thm:limitcase} lies in the following Lemma.

\begin{lemma}  \label{lem:lim1}
Under the assumptions of Theorem~\ref{thm:logedge}, for any $x\in \R^n$, the random variable $\mathrm{X}_N(\varphi_N(x)) + \log | E_N| $ converges in $L^1(\P_N)$ to $0$ as $N\to+\infty$. 
\end{lemma}

\begin{proof}
For simplicity, we denote $\varphi_N = \varphi_N(x)$ and  $\eta_N =|E_N|$. 
In particular, the last condition in \eqref{c4} implies that $|\varphi_N| \sim \eta_N$ and $\eta_N \to+\infty$ as $N\to+\infty$. 
 First observe that according to \eqref{def:X}, we have 
\[
\big| \mathrm{X}_N(\varphi_N) + \log \eta_N \big|  \le  \int \bigg| \log\bigg| \frac{\varphi_N- u}{\eta_N} \bigg|\bigg| 
\emp_N(\d u)
\]
so that by Proposition~\ref{prop:Wegner}, 
\begin{equation} \label{limit4}
\E_N\big[ \big| \mathrm{X}_N(\varphi_N) + \log \eta_N \big| \big]  \ll  \int \bigg| \log\bigg| \frac{\varphi_N- u}{\eta_N} \bigg|\bigg| 
\nu_\kappa(\d u)
\end{equation}
where $\nu_\kappa$ is the probability measure \eqref{def:mu0}. 
Since $\eta_N\to+\infty$, using the last last condition in \eqref{c1},  we have for any fixed $u\in \R^n$, 
\[
 \log\bigg| \frac{\varphi_N- u}{\eta_N} \bigg| \to \log| \upsilon| =0 .  
\]
Since $ \log^+\big| \frac{\varphi_N- u}{\eta_N} \big| \le \log\big(2+ \frac{|u|}{\eta_N}   \big) \le \log(2+|u|)$
where we used that $\eta_N\ge 1$ and $|\varphi_N| \le 2 \eta_N$ if $N$ is sufficiently large, 
by Lebesgue's dominated convergence theorem, this implies that  as $N\to+\infty$,
\[
 \int  \log^+\bigg| \frac{\varphi_N- u}{\eta_N} \bigg|  
\nu_\kappa(\d u)
\to 0 .
\]
In fact, by the same argument, we obtain that for any $k\in\N$, 
\begin{equation} \label{limit2}
 \int \bigg| k \wedge \log\bigg| \frac{\varphi_N- u}{\eta_N} \bigg|^{-1}  \bigg| \, 
\nu_\kappa(\d u)
\to 0 .
\end{equation}
On the other--hand, by a change of variables,
\[
 \int  \log\bigg| \frac{\varphi_N- u}{\eta_N} \bigg|^{-1} \1_{|\varphi_N- u| \le e^{-k} \eta_N}  \, 
\nu_\kappa(\d u)
=  \mathrm{C}_\kappa^{-1}  \int_{|z| \le e^{-k}} \log|z|^{-1} \eta_N^{n} e^{-\widetilde{V}_\kappa(\varphi_N +z \eta_N)}  \d z . 
\] 
From the conditions \eqref{c3}, there exists a constant $ c_\kappa \in\R$ such that
$\inf_{u\in\R^n}\big\{  \widetilde{V}_\kappa(u) - n \log |u|  \big\} \ge - c_\kappa$ so that  
\[
 \eta_N^{n} e^{-\widetilde{V}_\kappa(\varphi_N +z \eta_N)}  \le  e^{ c_\kappa} \big|\varphi_N/\eta_N +z \big|^{-n} . 
\]
Observe the RHS converges to $|v+z|^{-1} \le (1-e^{-1})^{-1}$ as $N\to+\infty$ uniformly for all $ |z| \le e^{-1}$.  This shows that for any $k\in\N$, 
\begin{equation} \label{limit3}
\limsup_{N\to+\infty}   \int  \log\bigg| \frac{\varphi_N- u}{\eta_N} \bigg|^{-1} \1_{|\varphi_N- u| \le e^{-k} \eta_N}  \,  \nu_\kappa(\d u)
\ll \int_{|z| \le e^{-k}} \log|z|^{-1} \d z 
\end{equation}
where the implied constant depends only on the parameter $\kappa$. 
Hence, since the RHS of \eqref{limit3} converges 0 as $k\to+\infty$, by combining this estimate with \eqref{limit2}, we conclude that as $N\to+\infty$
\[
\int \bigg| \log\bigg| \frac{\varphi_N- u}{\eta_N} \bigg|\bigg|  \nu(\d u) \to 0 . 
\]
From the estimate \eqref{limit4}, this completes the proof. 
\end{proof}

\begin{proof}[Proof of Theorem~\ref{thm:logedge}]
We proceed exactly as in the proof of Theorem~\ref{thm:limitcase}, except that we need to be slightly careful with uniformity of the limits as the interaction kernel $\g$ is not positive. 
Let us fix $k\in\N$  and a compact set $\mathcal{K} \subset \R^n$. 
First, by Lemma \ref{lem:lim1}, we obtain that for any fixed $\x\in\mathcal{K}^k$,  as $\beta N \to\gamma$, 
\begin{equation} \label{limit9}
  \E_{N-k} \big[  e^{-\beta (N-k) \sum_{i=1}^k \mathrm{X}_{N-k}( \varphi_N(x_i)) - k \beta (N-k) \log|E_N| }  \big]     \to 1. 
\end{equation}
Indeed if we let $\xi_N := \beta (N-k) \sum_{i=1}^k \big( \mathrm{X}_{N-k}( \varphi_N(x_i)) +\log|E_N| \big) $, then $\xi_N\to 0$ in probability and we claim that $e^{- \xi_N} \in L^r(\P_{N-k})$ for any $r\ge 1$. 
This last claim follows from \eqref{def:vartheta} and the estimate \eqref{UB2}. Namely, we have for any $r\ge 1$, 
 \[ \begin{aligned}
   \E_{N-k} \big[ e^{- r \xi_N} \big]  
   & \le   e^{r \beta (N-k) \sum_{i=1}^k \big( \vartheta( \varphi_N(x_i)) -  \log|E_N| \big) }   \E_{N-k} \big[ e^{r k \beta (N-k) \int \vartheta \d\emp_{N-k}}   \big]     \\
   & \le  C(\kappa)^{r k}   e^{r  \kappa \sum_{i=1}^k \big( \vartheta( \varphi_N(x_i)) -  \log|E_N| \big)_+ }   .
\end{aligned} \]
 Now, using the last last condition in \eqref{c1}, we see that $ \vartheta( \varphi_N(x_i)) -  \log|E_N| \to 0$  uniformly for all $\x\in\mathcal{K}^k$ since $\varphi_N(x_i) \sim |E_N|$ as $N\to+\infty$. 
 This implies that for any $k\in\N$ and $r\ge 1$,
 \begin{equation} \label{estXedge}
   \E_{N-k} \big[  e^{- r \beta (N-k) \sum_{i=1}^k \big( \mathrm{X}_{N-k}( \varphi_N(x_i)) + \log|E_N| \big) }  \big]  \le C(\kappa)^{2r k} 
 \end{equation}
 Formula \eqref{R2} for the $k^{th}$ correlation function still holds, so that using the asymptotics \eqref{limit9} and Corollary~\ref{cor:ratio},  we obtain for all $\x\in\mathcal{K}^k$, as $N\to+\infty$, 
\[
R_N^{(k)}(\x) =  e^{\beta N\log|E_N| }    e^{-\Ha_k( \varphi_N(x_1), \dots,  \varphi_N(x_k))}   \tfrac{N! \alpha_N^{-kn}}{\mathrm{L}_\gamma^k(N-k)!}   \big(1 +\o(1) \big) . 
\]
Note that we used that $N^{-1}\log|E_N| \to 0$ as $N\to+\infty$ for otherwise the first condition in \eqref{c4} cannot be satisfied since the potential $V(u)$ grows faster than $\kappa \log |u|$  for any $\kappa \ge 0$ (see the first condition \eqref{c3}). 
Since $\beta \log \alpha_N \to 0 $ (see the third condition in \eqref{c4}),  we have  for $x_i\neq x_j$,  $i,j\in\{1,\dots,k\}$, as  $\beta N \to\gamma$,
\[
\beta \, \g( \varphi_N(x_i),  \varphi_N(x_j))  =  \beta \big( \g( \psi(x_i),  \psi(x_j)) + \log \alpha_N   \big)  \to 0 . 
\]
Moreover, we verify by a Taylor expansion that the conditions \eqref{c4} and \eqref{c2} imply that as $N \to\infty$,
\begin{equation} \label{limit10}
\frac{N e^{-V(\varphi_{N}(x))+ \beta N\log|E_N|}}{ \mathrm{L}_\gamma \alpha_N^n}  \to \theta(x) = e^{-\upsilon \cdot \psi(x)} , 
\end{equation}
where the convergence is uniform for $x \in \mathcal{K}$. 
Then, this implies  that for almost all $\x  \in \mathcal{K}^k$, 
\[
\lim_{N\to+\infty} R_N^{(k)}(\x)    
=  \lim_{N\to+\infty}  \prod_{i=1}^k \frac{N e^{-V(\varphi_{N}(x))+ \beta N\log|E_N|}}{ \mathrm{L}_\gamma \alpha_N^n} =   \prod_{i=1}^k \theta(x_i) . 
\]
Finally, since the convergence \eqref{limit10} is uniform and using the estimate \eqref{estXedge} with $r=1$, by formula \eqref{R2}, we also obtain the uniform bound $ R_N^{(k)}(\x)  \le  C(\kappa)^{3 k}$ for all $\x\in\mathcal{K}^k$.
By Lemma \ref{lem:ppwcvg}, this show that the point process $\Xi_N = \sum_{j=1}^N  \boldsymbol\delta_{ \varphi_N^{-1}(x_j)}$ converges in distribution as $\beta N \to\gamma$ and $N\to+\infty$  to a (non--homogeneous) Poisson point process on $\R^n$ with intensity $\theta$.
\end{proof}

We easily deduce Theorem~\ref{thm:LEF} from Theorem~\ref{thm:logedge}. 
 
  \begin{proof}[Proof of Theorem~\ref{thm:LEF}]
It just suffices to verify that the potential $V(x) = |x|^\alpha$ satisfies the conditions~\eqref{c3}, \eqref{c4} and \eqref{c2}. 
  The sequence $\eta_N$ is constructed in such a way that 
  $\frac{Ne^{-V(\eta_N \upsilon)+ \beta N \log \eta_N}}{ \mathrm{L}_\gamma \alpha_N^{n}} \to 1 $
  with $\alpha_N = \alpha \eta_N^{\alpha-1}$. 
  The other conditions are easily verified as in the proof of Theorem~\ref{thm:REF}. 
  \end{proof}

\appendix\section{Appendix}

Before discussing the properties of the equilibrium measure and the law of large numbers, let us recall the following basic properties of the relative entropy \eqref{def:ent}. 

\begin{lemma} \label{lem:entropy}
For any $\nu\in\M(\X)$, the function $ \Hi(\cdot|\nu) \ge 0$ is lower semicontinuous and strictly convex. The level sets  $\{\Hi(\cdot|\nu)  \le t\big\}$  are compact for all $t\ge0$ and $ \Hi(\mu|\nu) = 0$ if and only if $\mu= \nu$.
\end{lemma}

Moreover, since $\mu_0 \ll \nu_\gamma$ according to \eqref{def:mu0} and \eqref{eq:mu}, we have 
 \[
 \Hi(\mu|\mu_0) = \Hi(\mu|\nu_\gamma) + \gamma \int \vartheta \d\mu   - \log(\mathrm{C}_\gamma/\mathrm{L}_0) , 
 \]
 where both sides could be $+\infty$. 
This implies that for any $\mu\in\M(\X)$ with $\displaystyle\int \vartheta \d\mu <+\infty$, 
\begin{equation} \label{F2}
 \F_\gamma(\mu) = \frac\gamma2  \En(\mu) + \Hi(\mu|\mu_0) + \log(\mathrm{C}_\gamma/\mathrm{L}_0)  , \qquad \gamma \ge 0.
 \end{equation}

\subsection{Large deviation principle and properties of the equilibrium measure}\label{sect:eq}

The goal of this section is to go over the proof of Proposition~\ref{prop:eq}. 
The methods that we use are classical (see e.g.~\cite{GZ18} and reference therein), but since our model is slightly more general than those previously studied in the literature, we go over the main steps of the proof.
Let us recall that under our assumptions, the free energy \eqref{def:F} is non--negative and it attains its minimum. At first, we verify that all minimizers have nice regularity properties and satisfy the self--consistent equation \eqref{eq:mu}. 
Then, we show that under $\P_N $, the empirical measure $\widehat{\mu}_N$ satisfies a large deviation principle whose rate function is (up to a constant) the free energy $\F_\gamma$. 
This is much stronger than the statement of  Proposition~\ref{prop:eq}. 
Finally, we review that for the Riesz and log gases on $\R^n$, the free energy has a unique minimizer. 

\begin{proposition} \label{prop:regularity}
Suppose that the  Assumptions~\ref{ass:2} hold and let $\mu_\gamma \in \M(\X)$ be a minimizer of the the free energy $\F_\gamma$. 
Then the potential  $\U^{\mu_\gamma}$ is continuous on $\X$ and  $\mu_\gamma$ has a density with respect to $\omega$ which  satisfies for all $x\in\X$, 
\begin{equation}  \label{EL3}
\mu_\gamma(x) = \mathrm{C}_\gamma^{-1} e^{- \gamma \U^{\mu_\gamma}(x) - V(x)} . 
\end{equation}
\end{proposition}

\begin{proof}
Since $\gamma \ge 0$ is fixed, we denote $\widetilde{V} = \widetilde{V}_\gamma $ and $\nu= \nu_\gamma$ to simplify notation.  
Without loss of generality, we can also assume that $\widetilde{V} \ge 0$. 
Then, observe  that if $\mu$ is absolutely continuous with respect to $\omega$, by  \eqref{def:ent} and \eqref{def:mu0}, we have 
\begin{equation} \label{EL0}
 \Hi(\mu|\nu) = \int  \mu \log \mu\, \d \omega + \int \widetilde{V}  \d \mu  + \log \mathrm{C}_{\gamma} , 
\end{equation}
where both sides could be $+\infty$. 
Moreover, under the Assumptions~\ref{ass:2}, we have $\widetilde\En(\nu) <+\infty$ so that
\[
\F_\gamma(\mu_\gamma) \le \F_\gamma(\nu) = \tfrac{\gamma}{2} \widetilde\En(\nu) <+\infty . 
\]
Consequently, $\mu_\gamma$ is absolutely continuous with respect to $\nu$ and we denote by $\mu_\gamma$  its density with respect to the reference measure $\omega$. 
%To obtain \eqref{EL3}, we follow the strategy from the proof of \cite[Proposition 2.1]{HL}. 

\paragraph{Step 1.}
 Let us verify that $\omega(x\in\S_V: \mu_\gamma(x) = 0) =0$. Otherwise, we can choose a bounded measurable set $\mathcal{A} \subset \{x\in \S_V : \mu_\gamma(x) = 0\}$ such that $\omega(\mathcal{A})>0$. Then, $\mu = (1-\epsilon) \mu_\gamma + \epsilon'\1_{\mathcal{A}} $ is a probability density for every $0<\epsilon<1$ where $\epsilon' = \epsilon/\omega(\mathcal{A})$ and we verify from \eqref{def:F} and  \eqref{EL0} that 
\[ \begin{aligned}
\F_\gamma(\mu) & = \frac{\gamma}{2} \widetilde\En(\mu) + (1-\epsilon) \int_{\mathcal{A}^c} \big(\widetilde{V} + \log \mu\big) \, \d\mu_\gamma
+ \epsilon \log(\epsilon')  +  \epsilon' \int_{\mathcal{A}} \widetilde{V}  \d \omega   + \log \mathrm{C}_\gamma  \\
& = \F_\gamma(\mu_\gamma) +  \epsilon \log \epsilon  +\O(\epsilon) , 
\end{aligned}\]
where we used that $\widetilde\En(\mu)  =  \widetilde\En(\mu_\gamma)  + \O(\epsilon) $ and $\displaystyle \int_{\mathcal{A}^c} \big(\widetilde{V} + \log \mu\big) \, \d\mu_\gamma =  \int \big(\widetilde{V} + \log \mu_\gamma\big) \, \d\mu_\gamma + \log(1-\epsilon) \mu_\gamma(\mathcal{A}^c)$. 
If $\epsilon>0$ is sufficiently small, this leads to $\F_\gamma(\mu) <  \F_\gamma(\mu_\gamma)$, which is clearly a contradiction. 
This shows that any minimizer $\mu_\gamma$ is equivalent to the reference measure $\mu_0$. 

\paragraph{Step 2.}  Let us derive the so--called \emph{Euler--Lagrange equation}. Fix a function $\phi:\X \to [-1,1]$ such that $\int \phi \d \mu_\gamma =0 $ and $\mu= (1+ \epsilon \phi) \mu_\gamma$ is a probability density for any $\epsilon \in (-1,1)$. Then, we verify that
\begin{equation} \label{EL1}
\frac{\d }{\d \epsilon} \F_\gamma(\mu) \Big|_{\epsilon=0}  = \int  \big(\gamma  \widetilde\U +\widetilde{V}  + \log \mu_\gamma \big) \phi\, \d\mu_\gamma 
\end{equation}
where $\displaystyle \widetilde\U(u) = \int  \widetilde\g(u,x) \mu_\gamma(\d x)$ and we used that 
\[
\iint \widetilde\g(u,v) |\phi(u)| |\phi(v)|   \mu_\gamma(\d u) \mu_\gamma(\d v) \le \widetilde\En(\mu_\gamma) <+\infty . 
\]
It easily follows from \eqref{EL1} and the fact that $\mu_\gamma$ is a minimizer of the free energy $ \F_\gamma$ that there exists a constant $\Upsilon_\gamma  \in \R$ such that 
\begin{equation} \label{EL2}
\mu_\gamma\big( x\in \X : \gamma  \widetilde\U(x) +\widetilde{V}(x) + \log \mu_\gamma(x) = \Upsilon_\gamma  \big)   = 0 .
\end{equation}

\paragraph{Step 3.}  
Let us show that the equation \eqref{EL2} holds for all $x\in\X$. 
Since $\widetilde\U \ge 0$, \eqref{EL2} implies that $\mu_\gamma(x) \ll \nu(x)$ for  $\omega$ almost every $x\in\X$, by {\bf Step 1.} 
Since $\displaystyle \int \vartheta(x) \nu(\d x) <+\infty $, this implies that for all $ u \in\X$, 
\begin{equation} \label{widetildeU}
\widetilde\U(u) =  \U^{\mu_\gamma}(u)   + \vartheta(u) + \int \vartheta(x) \mu_\gamma(\d x) 
\end{equation}
where both sides could (a priori) be $+\infty$. 
However, we can easily infer from the Assumptions~\ref{ass:1}.1)3) that the equilibrium potential $\U^{\mu_\gamma}$ is continuous. 
Indeed, if $u_\ell$ is any sequence which converges to $u$ in $\X$ as $\ell \to+\infty$, then for any $k,\ell \in \N$, 
\[ \begin{aligned}
\big|\U^{\mu_\gamma}(u_\ell) - \U^{\mu_\gamma}(u) \big| &  \ll 
\int \big|\g_k(u_\ell,x) - \g_k(u,x) \big|  \mu_\gamma(\d x) +
 \sup_{u\in\X} \bigg\{ \int \g^k(u,x) e^{-\widetilde{V} (x)} \omega(\d x)  \bigg\} \\
 &\ll \int \big|\g_k(u_\ell,x) - \g_k(u,x) \big|  \nu(\d x) +  \mathrm{c}_k  \bigg(\int e^{-\widetilde{V} (x)} \omega(\d x) \bigg)^{1/q}
\end{aligned}\]
where $q= \frac{p}{p-1} > 1$ and we used H\"older's inequality. Since the functions $\g_k$ are continuous and $|\g_k(x,u)| \le  k + \vartheta(x) + \vartheta(u)$, by the dominated convergence theorem, this implies that for any $k\in\N$, 
\[
\limsup_{\ell\to+\infty}\big|\U^{\mu_\gamma}(u_\ell) - \U^{\mu_\gamma}(u) \big|  \ll  \mathrm{c}_k. 
\]
As $ \mathrm{c}_k \to0$ as $k\to+\infty$, this proves the continuity of  $\U^{\mu_\gamma}$ and by \eqref{widetildeU}, the function $\widetilde\U$ is also continuous on $\X$. 
From \eqref{EL2} and since $\widetilde{V}$ is continuous, this shows that $\log \mu_\gamma$ equals to a continuous function almost everywhere. This establishes that  the equilibrium density $\mu_\gamma$ satisfies the following equation for all $x\in\X$, 
\[
\mu_\gamma(x) = e^{\Upsilon_\gamma- \gamma  \widetilde\U(x) -\widetilde{V}(x) } . 
\]
By \eqref{widetildeU} and since $\widetilde{V}(x)  = V(x) - \gamma \vartheta(x)$, we obtain that the density $\mu_\gamma$ satisfies the equation \eqref{eq:mu}. This completes the proof. 
Let us observe that  the  Assumptions~\ref{ass:2}, the equilibrium potential satisfies for all $u\in\X$,
\begin{equation} \label{Uest}
 \begin{aligned}
\U^{\mu_\gamma}(u) & \le 1+   \mathrm{c}_1 \\
- \U^{\mu_\gamma}(u) & \le \vartheta(u) +   c_\gamma \int \vartheta(x) \nu_\gamma(\d x) 
\end{aligned}
\end{equation}
where we have used  \eqref{def:vartheta}.
This justifies  the estimate \eqref{muk:tail}. 
\end{proof}

\paragraph{Large deviations.}
We now turn to the proof of the large deviation principle for the empirical measure.
Under the Assumptions~\ref{ass:1} $(\vartheta=0)$, the energy $\Ha_N(\x) \ge 0$ for any configuration $\x\in\X^N$ and the function \eqref{def:E} is lower semicontinuous. Then, the large deviation principle follows readily from  \cite[Corollary 1.3 and Section 2]{GZ18}. 
If the interaction kernel $\g$ is not bounded below, the situation is slightly more complicated and we give a few details for the convenience of the readers.

\begin{proposition}
Suppose that $\gamma > 0$  and that the  Assumptions~\ref{ass:2} hold. 
Under $\P_N$, the sequence of empirical measures $\emp_N$ satisfies a large deviation principle with speed $\beta N^2$  and rate function $ \gamma^{-1}\F_\gamma$ (up to a constant). 
\end{proposition}

\begin{proof}
Let us observe that by  \eqref{def:Gibbs} and \eqref{def:mu0}, we can rewrite
\[
\P_N[\d \x ]  : =  \frac{ e^{- \beta N^2 \W_N(\x)}}{\Z_N'}  \prod_{j=1}^N \nu_\gamma(\d x_j), 
\]
where $0<\Z_N'<+\infty$ and 
\[
 \W_N(\x) = N^{-2} \sum_{1\le i <j \le N} \widehat\g(x_i ,x_j) + \frac{\gamma-  \beta(N-1)}{\beta N^2} \sum_{1\le j \le N} \vartheta(x_j) .
\]
Then, according to  \cite[Theorem 1.2]{GZ18}, we just need to verify that $\widehat\W_\gamma$ is the \emph{positive temperature macroscopic limit} of $\W_N$. 
Fix a small $\epsilon>0$. We can assume that $N$ is large enough so that  $ (1-\epsilon) \gamma \le \beta(N-1) \le (1+\epsilon) \gamma$. 
Then, an immediate computation shows that for any $\mu\in\M(\X)$, 
\[ \begin{aligned}
\int \W_N(\x) {\textstyle \prod_{j=1}^N} \mu(\d x_j)  & = \frac{N-1}{2N} \iint  \widehat\g(u,v) \mu(\d u) \mu(\d v)  +  \frac{\gamma-  \beta(N-1)}{\beta N}  \int \vartheta(u)  \mu(\d u)   \\
&\le  \frac{1}{2} \iint  \widehat\g(u,v) \mu(\d u) \mu(\d v)  +  \frac{\epsilon}{1-\epsilon} \int \vartheta(u)  \mu(\d u)  . 
\end{aligned}\]
This implies that $\W_N$ satisfies the \emph{upper limit assumption} from \cite[(A.2)]{GZ19}: for any $\mu\in\M(\X)$, 
\[
\limsup_{N\to+\infty} \int \W_N(\x) {\textstyle \prod_{j=1}^N} \mu(\d x_j) \le  \gamma^{-1} \widetilde\En(\mu). 
\]
Similarly, we verify that
\[ \begin{aligned}
 \W_N(\x)
 & \ge  N^{-2} \sum_{1\le i <j \le N} \widehat\g(x_i ,x_j) \wedge \epsilon^{-1} - \frac{\epsilon}{(1-\epsilon)N}   \sum_{1\le j \le N} \vartheta(x_j) \\
&= \frac{1}{2} \iint  \widehat\g(u,v) \wedge \epsilon^{-1}\, \emp_N^{(\x)}(\d u) \emp_N^{(\x)}(\d v)
-   \frac{\epsilon}{1-\epsilon}   \int \vartheta(u) \emp_N^{(\x)}(\d u) 
- \frac{\epsilon^{-1}}{2N} ,
\end{aligned}\]
where the last term comes from the diagonal. 
Then, for any given $\mu \in \M(\X)$, if $\x^{(N)}$ is a sequence of configurations such that  the empirical measure $\emp_N^{(\x)} \to \mu$ as $N\to+\infty$, since $\widehat\g$ and $\vartheta$ are lower semicontinuous functions, by the Portmanteau Theorem,  the previous bound implies that
\[
\liminf_{N\to+\infty} \W_N(\x) \ge \frac{1}{2} \iint  \widehat\g(u,v) \wedge \epsilon^{-1} \mu(\d u) \mu(\d v) -   \frac{\epsilon}{1-\epsilon}   \int \vartheta(u)  \mu(\d u)  .
\]
As $\epsilon>0$ is arbitrary, we obtain that $\W_N$ satisfies the \emph{lower limit assumption} from \cite[(A.1)]{GZ19}: under the above conditions, 
\[
\liminf_{N\to+\infty} \W_N(\x) \ge  \widehat\En(\mu) .
\]
Since $\beta N \to \gamma \in (0,+\infty)$, by \cite[Corollary 1.3]{GZ19}, this completes the proof. 
\end{proof}

From the previous large deviation principle, it follows that if the free energy $\F_\gamma$ has a unique minimizer $\mu_\gamma \in \M(\X)$, then $\emp_N \to \mu_\gamma$ in probability  as  $N\to+ \infty$.
Together with Proposition~\ref{prop:regularity}, this completes the proof of Proposition~\ref{prop:eq}.

\medskip

\paragraph{Convexity of the free energy.}
Finally, let us quickly discuss the issue of uniqueness of the minimizer of the free energy \eqref{def:F} in the case of the Riesz and log gases. Once more, the following Lemmas are  classical. 

\begin{lemma} \label{lem:Rminimizer}
Let $\g_\s$ be as in \eqref{Riesz}, $\mathscr{Q} := \big\{ f\in L^1(\R^n) : \En(c|f|) <+\infty \text{ where } \displaystyle c = \int|f(x)| \d x\big\}$ and 
\vspace{-.3cm}
\[
\En(f) =  \iint \g_\s(x,z) f(z) f(x)\d z\d x    , \qquad f\in \mathscr{Q} . 
\] 
Then $\En$ is non--negative and strictly convex on $\mathscr{Q}$. 
Moroever, for any $\gamma\ge 0$, the free energy $\F_\gamma(\mu) = \frac\gamma2  \En(\mu) + \Hi(\mu|\mu_0)$ for $\mu \in\M(\R^n)$ has a unique minimizer $\mu_\gamma \in\M(\R^n)$.
\end{lemma}

\begin{proof}
The first claims regarding the positivity and convexity of  the energy $\En$ follow from \cite[Theorem 9.8]{LL01} -- This Theorem is stated and proved only in the Coulomb case $\s =n-2$ for $n\ge 3$, but these properties are true for any $n\in\N$ and $\s\in(0,n)$ with the same proof.
We already noticed that the function $\F_\gamma$ is lower--semicontinuous and positive on $\M(\R^n)$, so it attains its minimum and any minimizer $\mu$ is absolutely continuous with respect to $\mu_0$ (otherwise $\F_\gamma(\mu) = +\infty$). In particular, $\mu\in \mathscr{Q}$, so that by strictly convexity of both $\En$ and $\Hi(\cdot|\mu_0)$ (see Lemma~\ref{lem:entropy}) on $\mathscr{Q}$, we conclude that the minimizer of the free energy $\F_\gamma$ is unique. 
\end{proof}

In the following Lemma, we assume that $n=1\text{ or }2$,  $\g$ and $\vartheta$ are as in \eqref{logkernel}. 

\begin{lemma} \label{lem:Lminimizer}
Let $\mathscr{Q} = \bigg\{ f\in L^1(\vartheta) : \displaystyle\iint \big| \log|x-z|^{-1} f(z) f(x) \big| \d z\d x<+\infty\bigg\}$ and 
\[
\En(f) =  \iint \log|x-z|^{-1} f(z) f(x)\d z\d x    , \qquad f\in \mathscr{Q} . 
\] 
Then $\En(f) \ge 0$ for all $f\in\mathscr{Q}\cap\big\{ \int f \d x = 0 \big\}$ and $\En$ is strictly convex on $\mathscr{Q}\cap\M(\R^n)$. 
Moreover, under the Assumptions~\ref{ass:2}, the free energy \eqref{def:F} has a unique minimizer $\mu_\gamma \in\M(\R^n)$. 
\end{lemma}

\begin{proof}
It is well--known that if $f\in\mathscr{Q}$, $f$ has compact support and $\int f \d x= 0$, then  $\En(f) \ge 0$ and  
$\En(f) = 0$ if and only if $f=0$. 
For a proof of this claim, we refer to \cite[Lemma 6.41]{Deift99} when $n=1$ and \cite[Lemma 1.8]{ST97} when  $n=2$. 
By a standard approximation argument, we can show that this property holds true for all $f\in\mathscr{Q}$ with $\int f \d x= 0$. 
Moreover, by straightforward algebraic manipulations, we have for any probability density functions $f_0, f_1 \in\mathscr{Q}$  and $t\in[0,1]$,
\begin{equation} \label{convexity}
\En\big( (1-t)f_0+tf_1\big)-   (1-t)\En(f_0)- t\En(f_1) = - t(1-t) \En(f_1-f_0) . 
\end{equation}
This shows that the energy functional $\En$ is strictly convex on $\mathscr{Q}\cap\M(\R^n)$. 
Consequently, by \eqref{def:Ew}, the weighted energy functional  $\widetilde\En$ is also  strictly convex on $\mathscr{Q}\cap\M(\R^n)$. 
Recall that the free energy $\F_\gamma(\mu) = \frac\gamma2 \widetilde\En(\mu) + \Hi(\mu|\nu_\gamma)$ 
attains its minimum and that all minimizer(s) satisfy the equation \eqref{EL3}.
In particular, for all minimizer(s), the density $\mu_\gamma \ll \nu_\gamma$ --  see the estimates \eqref{Uest} at the end of the proof of Proposition~\ref{prop:regularity} -- so that by assumptions,  $\mu_\gamma\in\mathscr{Q}$.
From the strict convexity of $\widetilde\En$ and  $\Hi(\cdot|\nu_\gamma)$ (see Lemma~\ref{lem:entropy}), we conclude that the free energy  $\F_\gamma$ has a unique minimizer. 
\end{proof}

\subsection{Point processes on manifolds} \label{sect:pp}

\paragraph{Definition and correlation functions}
A (simple) point process is a random measure of the form $\Xi = \sum_{\lambda \in \Lambda} \boldsymbol\delta_{\lambda} $ 
where $\Lambda$ is a countable subset of $\X$ with no accumulation points.
%Note that it is possible to identify $\Xi$  with its support $\Lambda \subset \X$ -- hence the name point process. 
We refer to \cite{Kal} for the construction of such random processes. 
The law of a point process is characterized by its Laplace functional
\[
\psi(f) = \E\big[ e^{-\Xi(f)} \big]
 \qquad\text{for all Borel function } f: \X \to [0,+\infty).
\]
We can define the correlation functions $\big( R^{(k)} \big)_{k=1}^{+\infty}$ of the point process $\Xi$ through its Laplace functional:
\begin{equation} \label{def:psi}
\psi(f) =  1 + \sum_{k=1}^{+\infty} \frac{1}{k!} \int_{\X^k} \prod_{i=1}^k \big( e^{-f(x_i)} -1 \big) R^{(k)}(\d x_1,\cdots, \d x_k) . 
\end{equation}
A priori, the $k^{\rm th}$ correlation function $R^{(k)}$ is a measure on  the product space $\X^k$. It turns into a density function if it is absolutely continuous: 
$R^{(k)}(\d x_1,\cdots, \d x_k) = R^{(k)}(x_1,\cdots, x_k) \omega(\d x_1) \cdots \omega(\d x_k)$. 

\medskip

For instance, if  $\Xi_N = \sum_{j=1}^N \boldsymbol\delta_{\lambda_j} $  and $(\lambda_1,\dots, \lambda_N)$ has a (symmetric) joint distribution $\P_N$, then we verify that
\[
\psi_N(f) = \E_N\big[ e^{-\Xi_N(f)} \big]
=1 +\sum_{k=1}^{N} {N \choose k} \int_{\X^N} \prod_{i=1}^k \big( e^{-f(x_i)} -1 \big) \P_N[\d x_1,\dots, \d x_N] ,
\]
so that the correlation functions of the process $\Xi_N$ are given by for $k\in\N$,
\begin{equation} \label{correlations}
R_N^{(k)}(\d x_1,\cdots, \d x_k) = \1_{k\le N} \frac{N!}{(N-k)!} \int_{\X^{N-k}} \P_N[\d x_1,\dots,\d x_N] .
\end{equation}

Another important example is a Poisson point process. 
We will rely on the following definition. 

\begin{definition} \label{def:Poisson}
If $\theta$ is a Radon measure $\X$, we say that  $\Xi$ is a Poisson point process with intensity $\theta$ is its Laplace functional satisfies 
\[
\psi(f) =  \exp\bigg( \int_\X (e^{-f(x)}-1) \theta(\d x) \bigg) . 
\]
This immediately implies that the correlation functions of  $\X$ are given by 
$R^{(k)}(\d x_1,\cdots, \d x_k) = \prod_{i=1}^k \theta(\d x_i)$ for all $k\in\N$.
Moreover, we say that $\Xi$ is a homogeneous Poisson point process with intensity $\theta>0$ if its intensity $\theta$ is constant.  
\end{definition}

\paragraph{Weak convergence.} Let us also review the notion of convergence in distribution for point processes. 

\begin{definition}
We say that a sequence $\Xi_N$ of point processes $($with Laplace functional $\psi_N)$ on $\X$ converges in distribution to a point process $\Xi$ $($with Laplace functional $\psi)$ if for any function $f: \X \to [0,+\infty)$ continuous with compact support, $\psi_N(f) \to \psi(f)$  as $N\to+\infty$. 
\end{definition}

It is also easy to give a necessary condition for the convergence of point processes in terms of its correlation functions. 

\begin{lemma} \label{lem:ppwcvg}
For any $N\in\N \cup \{\infty\}$,  let $\Xi_N$ be a  point process on $\X$ with correlation functions $R_N^{(k)}(x_1,\dots , x_k)$. 
Then, $\Xi_N$ convergence in distribution to $\Xi_\infty$ if for any  $k\in\N$, 
$R_N^{(k)}(\x) \to R_\infty^{(k)}(\x)$ for almost all $\x \in \X^k$ and for any compact set $\mathcal{K}\subset \X$, 
\begin{equation} \label{cintegrability}
\sup_{N\in \N }  \sum_{k=1}^{+\infty} \frac{1}{k!} \int_{\mathcal{K}^k} R_N^{(k)}(x_1,\dots , x_k) \omega(\d x_1) \cdots \omega(\d x_k) <+\infty  . 
\end{equation}
\end{lemma}

\begin{proof}
The conditions of Lemma~\ref{lem:ppwcvg} imply that for any continuous function $f: \X \to [0,+\infty)$ with support in  $\mathcal{K}$,  $\psi_N(f) \to \psi(f)$  as $N\to+\infty$
by the dominated convergence theorem and since $|e^{-f}-1| \le \1_{\mathcal{K}}$. 
\end{proof}

Let us also observe that if $R_N^{(k)}(\x) \to R_\infty^{(k)}(\x)$ uniformly for all $\x$ in compact sets of $\X^k$ and the correlation functions $R_\infty^{(k)}$ are locally integrable, then the condition \eqref{cintegrability} is satisfied.

\paragraph{Change of variables.}
Let us record that if $U \subseteq \X$ is an open set and  $\varphi : U \to \R^n$ is a 1-1 map, then we can define a new point process $\widehat\Xi = \sum_{\lambda \in \Lambda \cap U} \boldsymbol\delta_{\varphi(\lambda)} $  on $\R^n$.
Almost surely, this point process is supported on the set $V =\varphi(U)$ and its correlation functions are given by the push--forward $\varphi_{\sharp} R^{(k)} \big|_U =:\widehat{R}^{(k)}$ for all  $k\ge 1$ (since $ R^{(k)}$ is a measure on $\X^k$, this notation means that $R^{(k)}$ is first restricted to $U^{\times k}$ and the push--forward is defined through the product map $\varphi^{\times k}$).
In particular, if $(\X,g)$ is a Riemannian manifold and $R^{(k)}$ have densities with respect to the volume form $\omega$, this implies that for almost all $\x\in (\R^n)^k$, 
\begin{equation} \label{trans}
\widehat{R}^{(k)}(\x) = R^{(k)}\big|_U\big( \varphi^{-1}(x_1),\cdots, \varphi^{-1}(x_k) \big) 
{\textstyle \prod_{i=1}^k} \mathscr{J}(x_i) \sqrt{\det g(\x)}
\end{equation}
where $\mathscr{J}$ is the Jacobian of $\varphi^{-1}$ and  the density $\widehat{R}^{(k)}$ is defined with respect to the Lebesgue measure on $(\R^n)^k$. 
Observe also that since $R^{(k)}\big|_U(u_1, \dots, u_k)  =  R^{(k)}(u_1, \dots, u_k) \1_{u_1,\dots, u_k \in U}$, the RHS of \eqref{trans} does not depend on how we extend the map $\varphi^{-1}$ outside of $V$ . 
Formula \eqref{trans} can be checked directly from \eqref{def:psi} and the change of variables formula and the fact that for all Borel function $f: \R^n \to [0,+\infty)$,
\[
\widehat\psi(f) = \E\big[ e^{-\widehat\Xi(f)} \big] =  \E\big[ e^{-\Xi(f|_V\circ \varphi)} \big] = \psi(f|_V\circ \varphi) . 
\]

\paragraph{Correlation of the local process.}
Using the notation \eqref{def:rho} and formula \eqref{correlations}, under $\P_N$, the correlation functions between the particles are given by $1_{k\le N} \frac{N!}{(N-k)!}  \rho_N^{(k)}(\u)$ for $\u\in \X^k$. 
Hence, by \eqref{def:Xi} and formula \eqref{trans}, the correlation functions (with respect to the Lebesgue measure) of the local point process $ \Xi_N = \sum_{x_j \in U}  \boldsymbol\delta_{N^{1/n} \varphi(x_j)}$
 which is obtained by zooming at a microscopic scale around the point $E$ are given by for any $k=1,\dots, N$, 
\begin{equation} \label{def:R}
R_N^{(k)}(\x) = \tfrac{N!}{(N-k)!} \rho_N^{(k)}\big|_U(u_1, \cdots  ,  u_k) {\textstyle \prod_{i=1}^k} \mathscr{J}(x_i) \sqrt{\det g(\x)} \, ,
\qquad \x \in (\R^n)^k . 
\end{equation}
where $u_i = \varphi^{-1}(x_i/N^{1/n})$ for $i=1,\dots, k$ and $\mathscr{J}$ denotes the Jacobian of the corresponding map.
Observe that because $\varphi$ is a normal coordinate chart, $u_i = {\rm Exp}_{E}(x_i/N^{1/n})$ and $g(0) = \mathrm{I}_n$, so that it holds as $N\to+\infty$,  
\[
N \mathscr{J}(x_i) \to  1 
\qquad\text{and}\qquad
\det g(\x) \to 1 
\]
uniformly for all $x_1, \dots, x_k$ in compact sets of $\R^n$. 
According to \eqref{def:R}, this shows that the correlation functions of the local process satisfies for any fixed $k\in\N$ as  $N\to+\infty$,  
\begin{equation} \label{R:asymp}
R_N^{(k)}(\x) =  \rho_N^{(k)}\Big|_U\big(\varphi^{-1}(x_1/N^{1/n}), \cdots  ,  \varphi^{-1}(x_k/N^{1/n})\big)
\big( 1+\o(1)  \big)
\end{equation}
uniformly for all $\x$ in a compact set of $(\R^n)^k$.

\paragraph{Edge scaling limit.} First, as a toy example, let us consider the case of $N$ i.i.d.~particles distributed according to a probability density function $\phi(u) = \zeta^{-1} e^{-V(u)}$ where $V \in C^2(\R^n \to [0,+\infty])$ 
and $\zeta>0$ is a normalizing constant. 
Then, the law of the particles is the Gibbs measure \eqref{def:Gibbs} with $\beta=0$ and the equilibrium density and $\phi$.  
This implies that for large $N$, the density of particles decay away from $0$ and we are interested in describing the local limit near the boundary of the droplet. Let us record the following simple Lemma.

\begin{lemma} \label{lem:Poissonedge}
Choose a sequence $(E_N)_{N\in\N}$ in $\R^n$ such that $\nabla V(E_N) = \alpha_N \upsilon_N$ 
where $\alpha_N>0$, $\|\upsilon_N\| =1$  and as $N\to+\infty$, 
\begin{equation} \label{cedge}
\frac{Ne^{-V(E_N)}}{\zeta \alpha_N^n} \to 1 
\qquad\text{and}\qquad \upsilon_N\to \upsilon . 
\end{equation}
Let $\varphi_N(x) = E_N + \alpha_N^{-1}\psi(x) $ where $\psi:\R^n \to \R^n$ is a continuous map with Jacobian 
$\mathscr{J}(\psi)=1$ and suppose that for any compact set $\mathscr{K} \subset \R^n$, as $N\to+\infty$,
\begin{equation} \label{cTaylor}
\alpha_N^{-2} \sup_{x \in \mathscr{K} }\left\| \nabla^2V\big(\varphi_N(x)\big) \right\| \to 0 . 
\end{equation}
Then, the point process $\Xi_N = \sum_{j=1}^N  \boldsymbol\delta_{ \varphi_N^{-1}(x_j)}$  obtained by zooming around the point~$E_N$  converges in distribution  to a Poisson point process with intensity $\theta(x) = e^{-\upsilon \cdot \psi(x)}$ with respect to the Lebesgue measure on $\R^n$. 
\end{lemma}

The proof of Lemma ~\ref{lem:Poissonedge} is an immediate consequence of Lemma~\ref{lem:ppwcvg} and the change of variables formula \eqref{trans}. 
Indeed, since we are in the Euclidean case and $\det\varphi_N = \alpha_N^{-n}$, by \eqref{correlations}, the correlation functions of the  process $\Xi_N$ are given exatcly by for all $k=1,\dots, N$, 
\[
R_N^{(k)}(\x) = \tfrac{N!}{(N-k)!}\zeta^{-k} \alpha_N^{- kn} {\textstyle\prod_{i=1}^k} \phi\big( \varphi_N(x_i) \big) , \qquad \x\in (\R^n)^k . 
\]
Moreover, since $V$ is $C^2$ at $E_N$, by a Taylor expansion and using the conditions \eqref{cedge}--\eqref{cTaylor}, we obtain the asymptotics as $N\to+\infty$, 
\[
\zeta^{-1} \alpha_N^{-n} N \phi\big( \varphi_N(x) \big)  \to  e^{-\upsilon \cdot \psi(x)} 
\]
which holds uniformly for all $x$ in compact sets of $\R^n$.
This shows that $R_N^{(k)}(\x)  \to {\textstyle\prod_{i=1}^k} e^{-\upsilon \cdot \psi(x_i)} $ uniformly on compact subsets of for  $(\R^n)^k$, so that the claim follows from Lemma~\ref{lem:ppwcvg}.

\clearpage

\end{document}